\documentclass[12pt]{amsart}

\title{The sparsity of character tables over finite reductive groups and its additive analogue}
\author{GyeongHyeon Nam$^1$}
\author{Anna Pusk\'{a}s$^2$}

\address{$^1$Department of Mathematics and Systems Analysis, Aalto University, Espoo 02150, Finland}
\email{\href{mailto:gyeonghyeon.nam@aalto.fi}{gyeonghyeon.nam@aalto.fi}}

\address{$^2$School of Mathematics \& Statistics,  University of Glasgow,  G12 8QQ, Glasgow, Scotland}
\email{\href{mailto:anna.puskas@glasgow.ac.uk}{anna.puskas@glasgow.ac.uk}}

\usepackage{fullpage} 
 \usepackage[alphabetic]{amsrefs}
\usepackage{amssymb}
\usepackage{amsmath}
\usepackage{mathtools} 
\usepackage{xcolor}
\usepackage{colonequals}
\usepackage{amsrefs}
\usepackage[linktoc=all]{hyperref}
\usepackage{textgreek}
\usepackage{tikz-cd}
\usepackage{mleftright}
\usepackage{wrapfig}
\usepackage{MnSymbol}
\usepackage{tikz}
\usetikzlibrary{decorations.markings}

\makeatletter
\newcommand{\vo}{\vec{o}\@ifnextchar{^}{\,}{}}
\makeatother

\theoremstyle{plain}
\newtheorem{thm}{Theorem}[subsection]
\newtheorem{lem}[thm]{Lemma}
\newtheorem{prop}[thm]{Proposition}

\newtheorem{conj}[thm]{Conjecture}

\newtheorem{cor}[thm]{Corollary}
\newtheorem{ques}[thm]{Question}
\theoremstyle{definition}
\newtheorem{defe}[thm]{Definition}
 
\theoremstyle{remark}
\newtheorem{rem}[thm]{Remark}

\definecolor{red}{rgb}{1,0,0}
\definecolor{orange}{rgb}{1,0.5,0}
\definecolor{purple}{rgb}{.5,.2,.8}
\definecolor{blue}{rgb}{.2,.2,.8}
\definecolor{green}{rgb}{.4,.6,.4}

\def\bes{\begin{equation*}}  \def\ees{\end{equation*}} 
\def\bi{\begin{itemize}}   \def\ei{\end{itemize}}
\def\ba{\begin{eqnarray}} \def\ea{\end{eqnarray}}    
\def\bl{\begin{align}}    \def\el{\end{align}}       
\def\bls{\begin{align*}}    \def\els{\end{align*}}

\newcommand{\bC}{\mathbb{C}}

\newcommand{\chG}{\check{G}}

\newcommand{\GL}{\mathrm{GL}}

\newcommand{\Z}{\mathbb{Z}}

\newcommand{\fg}{\mathfrak{g}}

\newcommand{\ft}{\mathfrak{t}}

\newcommand{\N}{\mathbb{N}} 
\newcommand{\R}{\mathbb{R}} 

\begin{document}

\begin{abstract}
We consider the proportion of zero entries in the character table of a sequence of reductive groups over a finite field. We prove an asymptotic lower bound when the reductive group is fixed and the size of the finite field increases. Furthermore, we prove that when considering a sequence of reductive groups with increasing semisimple rank, the proportion is asymptotically one. 
We also establish an additive analogue of this phenomenon in the context of a fixed reductive Lie algebra. \end{abstract}


\maketitle

\tableofcontents

 \section{Introduction}

In the representation theory of finite groups, determining all character values is a natural and fundamental task, yet it remains a challenging one. The work of Deligne and Lusztig provides deep results on this problem in the context of finite reductive groups. However, as far as we know, there are some missing links to compute every character value explicitly. 
In this paper, we study the character tables of finite reductive groups from the viewpoint of arithmetic statistics. 

Similar work was carried out in \cite{larsen2020sparsity} and \cite{gallagher2022many}, where the authors proved results using characteristic polynomials. Our method is based on Deligne-Lusztig theory and regular semisimple elements. 
This approach has two advantages: it is independent of Cartan type, and it allows us to also consider an additive analogue of the question. The additive analogue uses the Fourier transform on characteristic functions of the finite Lie algebra.
This allows us to study reductive groups and their Lie algebras together, and the proofs in the two settings follow along very similar lines.

 \subsection{Overview}\label{ss:overview}  
We first provide an overview of this paper and then explain our main results in the following subsections.

\subsubsection{}\label{sss:lowerboundintro} We consider a lower bound for the following limit, where $G$ is a fixed connected reductive group over a finite field:
\begin{equation}\label{eq:introbrief}
 \underset{q\rightarrow \infty}{\lim}\frac{|\{(\chi,O) \in \widehat{G(\mathbb{F}_q)}\times [G(\mathbb{F}_q)]\,|\, \chi(O)=0\}|}{|\widehat{G(\mathbb{F}_q)}|\times |[G(\mathbb{F}_q)]|},
\end{equation}
where $\widehat{H}$ is the set of irreducible complex characters of a finite group $H$ and $[H]$ the set of conjugacy classes. For detailed definition of the sequence, limit and result, please see Theorem \ref{thm:fixedgroup}.
We provide some examples as follows.
\begin{enumerate}
\item[(i)] When $G=\GL_2$, we have that 
\[
 \underset{q\rightarrow \infty}{\lim}\frac{|\{(\chi,O) \in \widehat{G(\mathbb{F}_q)}\times [G(\mathbb{F}_q)]\,|\, \chi(O)=0\}|}{|\widehat{G(\mathbb{F}_q)}|\times |[G(\mathbb{F}_q)]|}=  \underset{q\rightarrow \infty}{\lim}\frac{q^2-2q+2}{2(q+1)^2}=\frac{1}{2}.
\]
\item[(ii)] When $G=\GL_3$, we have that 
\[
 \underset{q\rightarrow \infty}{\lim}\frac{|\{(\chi,O) \in \widehat{G(\mathbb{F}_q)}\times [G(\mathbb{F}_q)]\,|\, \chi(O)=0\}|}{|\widehat{G(\mathbb{F}_q)}|\times |[G(\mathbb{F}_q)]|}=  \underset{q\rightarrow \infty}{\lim}\frac{11q^4-2q^3+14q^2-45q-18}{18q^2(q+1)^2}=\frac{11}{18}.
\]
\end{enumerate}
We give a lower bound of the limit (equation \eqref{eq:introbrief}) using the Weyl group of $G$.  
 
\subsubsection{}\label{sss:111}
Our next result states, in some sense, that the nonzero entries in the character table are asymptotically sparse. This means that if we take a sequence of finite reductive groups whose semisimple rank approaches infinity, then the limit of the proportion of vanishing character values is one. For precise details please refer to Theorem \ref{thm:main}. Note that it would be an interesting problem to compute this limit for the sequence of symmetric or alternating groups from \cite{larsen2020sparsity}.

 \subsubsection{}
 To study problems in Section \ref{sss:lowerboundintro} and \ref{sss:111}, we use Deligne-Lusztig theory in the following way. We consider irreducible characters that result from induction from a maximal torus. These characters vanish on regular semisimple elements of the group in other maximal tori, as long as the two maximal tori are not conjugate. Counting the zero entries corresponding to such pairs is the main content of the proof. Conjugacy classes of maximal tori of a finite reductive group are in bijection with conjugacy classes in the Weyl group. Therefore, the computation is ultimately reduced to an estimate on the probability that two elements of the Weyl group are in the same conjugacy class. This probability is computed in a result by \cite{blackburn2012probability}.

\subsubsection{}  It is well-known that a reductive group and its Lie algebra interact closely, so we may consider the same problem on the corresponding finite Lie algebra. This approach becomes possible by using the Fourier transform on finite Lie algebras, cf. \cite{Letellier05}. This transform produces an analogue of irreducible complex characters of the finite reductive group. This parallel is exploited in many places, for example, \cite{HLRV, KNWG}.
We give an additive analogue of Section \ref{sss:lowerboundintro} over a fixed Lie algebra using this Fourier transform as detailed in Theorem \ref{thm:fixedliealgebra}.

Finally, we consider an additive analogue of Section \ref{sss:111} with a sequence of finite Lie algebras, see Conjecture \ref{conj:additiveanalogue}. We anticipate that this problem can be solved by computing uniform bounds of the number of adjoint orbits and regular semisimple orbits of any finite Lie algebra. A detailed explanation is given in Section \ref{sss:conjectureexplanation}.

\subsubsection{Notation}

Let us introduce some notation. Let $G$ be a (untwisted) connected reductive group, $W(G)$ the Weyl group of $G$, $\tau(G)$ the number of conjugacy classes in $W(G)$, and $w_1, \ldots , w_{\tau(G)}\in W(G)$ representatives of distinct conjugacy classes in $W(G)$. If the group is clear in the context, we drop $(G)$. For a finite group $H$, recall that we denote the set of conjugacy classes of $H$ as $[H]$ and the set of irreducible (complex) characters of $H$ as $\widehat{H}$.

\subsection{A fixed connected reductive group}\label{ss:fixedgroupintro}In this subsection, we will fix a (untwisted) connected reductive group $G$ and increase the size of the finite field. 
Then the following theorem is our first main result.

\begin{thm}\label{thm:fixedgroup}
Let us consider the following cases.

\begin{enumerate}
\item[(F1)] We have a group defined as a group scheme ${\mathbb{G}}$ over $\Z $ and we consider finite fields ${\mathbb{F}}_n,$ with $|{\mathbb{F}}_n|=q_n$ for $n=1,2,3,\ldots $ such that $q_n\rightarrow \infty $ as $n\rightarrow \infty$.
Here, $q_n$ is a power of a prime $p_n$ for all $n$. Then let us take $k_n=\overline{\mathbb{F}_{p_n}}$, and we consider the sequence of tensor products $\{G_n^{F_n}\}_{n=1}^\infty$, where $G_n=\mathbb{G}\otimes_{\mathbb{Z}} k_n$ a group scheme over $k_n$ and the Frobenius map $F_n=F_{q_n}$.  

\item[(F2)]  We have a connected reductive group $G$ defined over a field $k,$ where $k$ is algebraically closed, ${\mathrm{char}}\ k=p>0,$  $q=p^e$ for some $e\in \Z_{>0}$, and $G^{F_{q^n}}=G^{({F_q}^n)}$ for $F_q^n\colonequals \underset{n\text{-times}}{\underbrace{F_q\circ \cdots \circ F_q}}$. Then we consider the sequence $\{G^{F_{q^n}}\}_{n=1}^\infty$.
\end{enumerate}

We assume that $G$ and $\mathbb{G}$ are a (untwisted) connected reductive group and group scheme, whose centres are connected and whose derived subgroups are simply connected. 
Then when every characteristic $p_n$ and $p$ are very good for $G_n$ and $G$, we have the following results:
    \[    \underset{n\rightarrow \infty}{\lim}\frac{|\{(\chi,O) \in \widehat{G_n^{F_n}}\times [G_n^{F_n}]\,|\, \chi(O)=0\}|}{|\widehat{G_n^{F_n}}|\times |[G_n^{F_n}]|} \geq 1- \sum_{i=1}^{\tau }\frac{1}{|C_{W }(w_{i})|^2}
    \]
    and  
\[
    \underset{n\rightarrow \infty}{\lim}\frac{|\{(\chi,O) \in \widehat{G^{F_q^n}}\times [G^{F_q^n}]\,|\, \chi(O)=0\}|}{|\widehat{G^{F_q^n}}|\times |[G^{F_q^n}]|} \geq 1-\sum_{i=1}^{\tau}\frac{1}{|C_{W}(w_{i})|^2}.
  \]
       Note that in case $\mathrm{(F1)}$, the Weyl groups of $G_n$ are all same since they are determined by $\mathbb{G}$, so let us denote this Weyl group as $W $ and the number of its conjugacy class as $\tau$.
\end{thm}

\subsection{The sparsity of character tables}\label{ss:mainthmintro}

Now, we take a sequence of connected reductive groups such that the semisimple ranks tend to infinity.
Then we have the following result. 
\begin{thm}\label{thm:main}Let $\{G_n\}_{n=1}^\infty$ be a sequence of (untwisted) connected reductive groups defined over an algebraically closed field $k_n$ of positive characteristic with the Frobenius map $F_n\colonequals F_{q_n}$ (where $q_n$ is a power of a prime $p_n$), whose centre of $G_n$ is connected and whose derived subgroup is simply connected for every $n$.  	

	For almost all n of $\{G_n\}_{n=1}^\infty$, let us assume that the size of $k_n^{F_n}=q_n$ satisfies the following:
	\begin{enumerate}
\item[(S1)] There is a function $f:\Z_{>0}\rightarrow \R_{>1}$ such that $r\cdot f(r)$ is increasing on $\Z_{>0},$ $\lim_{r\rightarrow \infty }f(r)=\infty $ and $q_n\geq l_n\cdot f(l_n)$ for $l_n=\mathrm{rank}\ [G_n,G_n].$
\item[(S2)] For every maximal torus $T$, the finite torus $T^{F_n}$ is non-degenerate (for the definition, see \cite[\S3.6]{carter1985finite}) and contains regular elements. 
\item[(S3)] The characteristic $p_n$ is very good for $G_n$ and $q_n> |l_n!|_{p_n}.$ (Here $|l_n!|_{p_n}$ is the highest power of $p_n$ that divides $l_n!$)
	\end{enumerate}
Then we have
    \begin{equation}\label{eq:mainthm1}
    \underset{n \rightarrow \infty}{\mathrm{lim}}\frac{|\{(\chi,O) \in \widehat{G_n^{F_n}}\times [G_n^{F_n}]\,|\, \chi(O)=0\}|}{|\widehat{G_n^{F_n}}|\times |[G_n^{F_n}]|}= 1.
\end{equation}
\end{thm}

\begin{rem} Note that for any Cartan type, there exists a reductive group which has connected centre and simply connected derived subgroup, see \cite[Theorem 1.21]{taylor}.  The assumptions on $G_n$ in Theorem \ref{thm:main} are imposed to count the number of regular semisimple elements and specific irreducible characters carefully. Furthermore, the condition (S1) is important to show equation \eqref{eq:mainthm1} via $\epsilon$-$\delta$ method.

\end{rem}

  \begin{rem}
  Note that if we consider the sequence 
$\{G_n\}_{n=1}^\infty$
  of reductive groups whose ranks tend to infinity (instead of their semisimple ranks), then we have some exceptional cases. For example, let us assume that almost all $G_n$ are the product of a large torus and a fixed reductive group. Then we would have that equation \eqref{eq:mainthm1} is not true since all cells of the character table of a torus are non-zero.  
 \end{rem}

\subsection{The Lie algebra version} 
Now, let us consider additive versions of the previous results.
When we consider a reductive Lie algebra $\mathfrak{g}=\mathrm{Lie}(G)$, we have an analogue of irreducible characters of $G^F$ via the Fourier transforms $\mathcal{F}$ on $G^F$-invariant functions on $\fg^F$, cf. Section \ref{ss:frobeniustransform} (for more detailed explanation, please see \cite{Letellier05}). 

Let us denote $[\mathfrak{g}^F]$ as the set of adjoint orbits in $\mathfrak{g}^F$ under $G^F$-action. For each $\mathcal{O}_i \in [\mathfrak{g}^F]$, we have the characteristic function on $\mathcal{O}_i$, denoted by $1_{\mathcal{O}_i}^{G^F}\in \mathbb{C}[\mathfrak{g}^F]^{G^F}$, where $\bC[\fg^F]^{G^F}$ is the space of $G^F$-invariant functions on $\fg^F$. If the Frobenius map is obvious, then we denote this characteristic function as $1_{\mathcal{O}_i}^{G}$.
Now, let us denote the set of $\mathcal{F}(1_{\mathcal{O}_i}^G)\in \mathbb{C}[\mathfrak{g}^F]^{G^F}$ by $\mathfrak{F}$, i.e. 
\[
\mathfrak{F} \colonequals \left\{ \mathcal{F}(1_{\mathcal{O}_i}^G)\,|\, i=1,2, \ldots, |[\mathfrak{g}^F]| \right\}.
\]
As noted before, the elements $ \mathcal{F}(1_{\mathcal{O}_i}^G)$ have similar properties with irreducible characters of $G^F$, for example, please see \cite{lehrer1996space, KNWG}.

\subsubsection{A fixed Lie algebra}  
We present an additive analogue of Theorem \ref{thm:fixedgroup}, in other words, we fix a reductive Lie algebra $\fg$ and increase the size of the finite field.

\begin{thm}\label{thm:fixedliealgebra} Let us consider the following conditions, which are Lie algebra analogue of Theorem \ref{thm:fixedgroup}. 

\begin{enumerate}
\item[(A1)] We have a group defined as a group scheme ${\mathbb{G}}$ over $\Z $ and we consider finite fields ${\mathbb{F}}_n,$ with $|{\mathbb{F}}_n|=q_n$ for $n=1,2,3,\ldots $ such that $q_n\rightarrow \infty $ as $n\rightarrow \infty$.
Here, $q_n$ is a power of a prime $p_n$ for all $n$. Then let us take $k_n=\overline{\mathbb{F}_{p_n}}$, and we consider the sequence of tensor products $\{\fg_n^{F_n}\}_{n=1}^\infty$, where $\fg_n=\mathrm{Lie}(\mathbb{G}\otimes_{\mathbb{Z}} k_n)$ with the Frobenius map $F_n=F_{q_n}$.  

\item[(A2)]   We have a connected reductive group $G$ defined over a field $k,$ where $k$ is algebraically closed, ${\mathrm{char}}\ k=p>0,$ $q=p^e$ for $e\in \Z_{>0}$, and $G^{F_{q^n}}=G^{({F_q}^n)}$ for $F_q^n\colonequals \underset{n\text{-times}}{\underbrace{F_q\circ \cdots \circ F_q}}$. Then we consider the sequence $\{\fg^{F_{q^n}}\}_{n=1}^\infty$ for $\fg=\mathrm{Lie}(G)$.
\end{enumerate}

We assume that all characteristic $p_n$ and $p$ are very good and {regular} (as defined in \cite[page 242]{lehrer1992rational}) for $G_n$ and $G$. 
Then we have that 
\[
    \underset{q \rightarrow \infty}{\mathrm{lim}}\frac{|\{(\mathcal{F}_n(1_{\mathcal{O}_i}^{G_n}),O) \in \mathfrak{F}_n\times [\fg_n^{F_n}]\,|\, \mathcal{F}(1_{\mathcal{O}_i}^{G_n})(O)=0\}|}{|\mathfrak{F}_n|\times |[\fg_n^{F_n}]|}
\geq 1-\sum_{i=1}^{\tau } \frac{1}{|C_{W }(w_{i})|^2}
  \]
    and
    \[
    \underset{n\rightarrow \infty}{\lim}\frac{|\{(\mathcal{F}(1_{\mathcal{O}_i}^{G^{F_q^n}}),O) \in \mathfrak{F}_{q^n}\times [\fg^{F_q^n}]\,|\, \mathcal{F}(1_{\mathcal{O}_i}^{G^{F_q^n}})(O)=0\}|}{|\mathfrak{F}_{q^n}|\times |[\fg^{F_q^n}]|}
\geq 1- \sum_{i=1}^{\tau} \frac{1}{|C_{W}(w_{i})|^2},
      \]
where $\mathcal{F}_n$ is the Fourier transform on $\fg_n^{F_n}$ or $\fg^{F_q^n}$ and $\mathfrak{F}_{q^n} \colonequals \left\{ \mathcal{F}(1_{\mathcal{O}_i}^{G^{F_q^n}})\,|\, i=1,2, \ldots , |[\fg^{F_q^n}]| \right\}$.
\end{thm}

\subsubsection{}
Now, we present an additive analogue conjecture of Theorem \ref{thm:main} as follows: 
\begin{conj}\label{conj:additiveanalogue}
Let $\left\{ \fg_n\right\}_{n=1}^{\infty}$ be a sequence of reductive Lie algebras for $\mathfrak{g}_n=\mathrm{Lie}(G_n)$, defined over an algebraically closed field $k_n$ of positive characteristic with the Frobenius map $F_n= F_{q_n}$ (where $q_n$ is a power of a prime $p_n$) such that whose semisimple ranks tend to $\infty$. Let us consider the following set
\[
\mathfrak{F}_n \colonequals \left\{ \mathcal{F}(1_{\mathcal{O}_i}^{G_n})\,|\, i=1,2, \ldots, |[\mathfrak{g}_n^{F_n}]| \right\}.
\]
 Furthermore, let us assume that the characteristic of $p_n$ is very good and {regular} for $\fg_n$ with $q_n>|\Phi(\fg_n)|$ for almost all $n$, where $|\Phi(\fg_n)|$ is the root system of $G_n$.  Then we have
\[
    \underset{n \rightarrow \infty}{\mathrm{lim}}\frac{|\{(\mathcal{F}_n(1_{\mathcal{O}_i}^{G_n}),O) \in \mathfrak{F}_n\times [\fg_n^{F_n}]\,|\, \mathcal{F}_n(1_{\mathcal{O}_i}^{G_n})(O)=0\}|}{|\mathfrak{F}_n|\times |[\fg_n^{F_n}]|}= 1.
\]
\end{conj}
We anticipate that the proof of this conjecture would be similar to the proof of Theorem \ref{thm:main}. The one missing milestone to prove this conjecture is given in Section \ref{sss:conjectureexplanation} (in the form of Question \ref{que:missingstep-conj6}), and we believe that solving this conjecture would be an interesting project. In Section \ref{sssect:porc} we also mention a question about counting elements of a given type in the additive context. This is a would-be analogue of results from \cite{deriziotis, BK23}.

\section{Preliminaries}\label{sect:prelim}

In this section, we introduce basic materials and well-known facts which will be mainly used to prove our main results.

\subsection{Assumptions and notation}\label{subs:assumptions+notation}In this subsection, we introduce the assumptions and notation that will be used throughout this paper. These were briefly mentioned in the Introduction.

In this paper, we mean that a connected reductive group $G$ is untwisted, split, its centre $Z(G)$ is connected, and the derived subgroup $[G,G]$  simply connected. Assume $G$ is defined over an algebraically closed field $k$ of positive characteristic $p$, and the characteristic is very good for $G$ in the sense of \cite[\S2.1]{bate}. Let $F:k\rightarrow k$ be a Frobenius automorphism with $|k^F|=q.$  

Let $T$ be a maximal $F$-stable split torus of $G,$ and $r$ the rank of $G,$ i.e. $T\simeq (k^{\ast})^r$ and $T^F\simeq ({\mathbb{F}_q^{\ast}})^r.$
Let us assume that every maximal torus of $G^F$ is non-degenerate, cf. \cite[Proposition 3.6.6]{carter1985finite}.
 Let $W$ be the Weyl group of $G$ with respect to $T$, i.e. $W\simeq N_G(T)/C_G(T)=N_G(T)/T.$ 
We assume $G$ is untwisted, i.e. the $F$-action on the Weyl group $W_G=N_G(T)/T$ is trivial, cf. \cite[\S1.6]{geck2020character}. In particular, $F$-conjugacy classes of $W$ are just conjugacy classes of $W;$
let $\tau(W)$ denote their number. If the Weyl group is clear from the context, we denote $\tau(W)$ by $\tau$.

Let us denote the root datum of a connected reductive group $G$ as $(X,\Phi,Y,\Phi^\vee)$, and $\check{T}=\mathrm{Spec}\ k[X]$ the dual torus of $T$. Then we have the Langlands dual group of $G$ over $k$, denoted by $\chG$. In other words, $\chG$ is the connected reductive group over $k$ with a split maximal torus $\check{T}$ and root datum $(Y,\Phi^\vee,X,\Phi)$.


\subsection{$F$-stable tori}\label{subs:F-stable-tori}
The $F$-stable maximal tori of $G$ are all conjugate by elements of $G,$ but not necessarily by elements of $G^F.$ 
Here we summarise facts about $G^F$-conjugacy classes of $F$-stable tori of $G.$ The reader may wish to consult \cite[\S 3.3]{carter1985finite} for further details on this section. 

\begin{prop}\label{prop:FstableTorifacts}
	Let $G,$ $T,$ $W$ and $F$ be as in Section \ref{subs:assumptions+notation}. Then $G^F$-conjugacy classes of $F$-stable maximal tori of $G$ are in bijection with conjugacy classes of $W.$ If $w\in W$ and $T'=T[w]$ is an $F$-stable maximal torus of $G$ corresponding to the class of $W,$ then the quotient group $N_G(T[w])^F/T[w]^F$ is isomorphic to $C_W(w)=\{x\in W\,|\, x^{-1}wx=w \}$. The order $ |T[w]^F|$ is a monic polynomial of $q$ of degree $r.$
\end{prop}
\begin{proof}
The correspondence between $G^F$-conjugacy classes of $F$-stable maximal tori of $G$ and conjugacy classes of $W$ is \cite[Proposition 3.3.3]{carter1985finite}. 	
	To illustrate the bijection, recall that for $g\in G$ the maximal torus $gTg^{-1}$ is $F$-stable if and only if $g^{-1}F(g)\in N_G(T)$ \cite[Proposition 3.3.2]{carter1985finite}, the map $N_G(T)\rightarrow W$ yields an element $w.$ We then say that $T[w]:=gTg^{-1}$  is obtained from $T$ by twisting with $w.$ The $W$-conjugacy class of $w$ determines $T[w]$ up to conjugation by $G^F$. Observe that $C_G(T)=T$ \cite[p.28]{carter1985finite} and the above implies that $T[w]$ is well-defined from $T.$ 

    The second statement is  \cite[Proposition 3.3.6]{carter1985finite}.
	
	It follows from \cite[Proposition 3.3.5]{carter1985finite} that $|T[w]^F|=\chi(q)$ where $\chi(x)$ is the characteristic polynomial of $w$ acting on the real vector space spanned by the cocharacter lattice of $T=T_1,$ which has dimension $r.$
\end{proof}

We fix the following notation for the remainder of the paper. We pick representatives $w_1=e_W,\ \ldots ,\ w_{\tau}$ of conjugacy classes in $W$, and set $T_i:=T[w_i]$ to be the tori obtained from $T$ by twisting with $w_i$ for $1\leq i\leq \tau.$ (Note that $T=T_1.$) Furthermore, we set $c_i:=|C_W(w_i)|$ for the sizes of the centraliser subgroups in the Weyl group. Since the conjugacy classes of size $\displaystyle \frac{|W|}{c_i}$ partition $W,$ we have 
\begin{equation}\label{eq:sumofreciprocalci}
	\sum_{i=1}^{\tau} \frac{1}{c_i}=1.
	\end{equation}

\subsection{Regular semisimple elements}

As briefly introduced in Section \ref{ss:overview}, we are interested in regular semisimple elements of $G^F.$ In particular, we wish to count the number of classes of these in a given $F$-stable torus of $G^F.$ Recall the definition of a regular semisimple element.


\begin{defe}\label{def:regularsemisimple}
	The semisimple element $g\in G$ is {\emph{regular}} if $\dim C_G(g)=r.$ 
\end{defe}
This means that a semisimple element is regular if its centraliser has the least possible dimension \cite[\S 2.2, 2.3]{humphreys1995conjugacy}. 
We collect some useful information in the following. 
\begin{prop}\label{prop:properties-counts-of-regular-semisimple}
	Let $G$ be as in Section \ref{subs:assumptions+notation}, and let $w_i$, $c_i$, and $T_i$ (for $i=1,2, \ldots , \tau$) be as in Section \ref{subs:F-stable-tori}. 
	Then each $T_i^F$ is also a representative of a conjugacy class of maximal tori in $G^F$.
 	  The number $f_i$ of regular elements in $T_i^F$ is a monic polynomial in $q$ of degree $r.$ The semisimple conjugacy classes of $G^F$ are in bijection with $F$-stable orbits in $T/W.$ 
	 The number of $G^F$-conjugacy classes of regular semisimple elements in $T_i^F$ is $\displaystyle \frac{f_i}{c_i}.$ 
\end{prop}

\begin{proof}
The first statement follows from \cite[Proposition 3.6.2]{carter1985finite}.
The second statement can be proved by \cite[Lemma 2.3.11]{geck2020character} by considering $f_i \in \mathbb{C}[q]$.
Another way to prove this fact is that applying the M\"obius inversion on the closed subsystems of $\Phi$, as noted in \cite[\S 8.9]{humphreys1995conjugacy}, see also \cite{lehrer1992rational}.

The statement about the bijection between $G^F$-classes of semisimple elements and $F$-stable orbits in $T/W$ is \cite[Proposition 3.7.3]{carter1985finite}. It is clear from the proof of \cite[Proposition 3.7.1]{carter1985finite} that two elements of $T$ are conjugate by an element of $G^F$ if and only if they are $W$-conjugate.

Let us consider the last statement by taking two regular elements $t,t'\in T_i^F$ that are conjugate via an element $g\in G^F,$ so that $t'=gtg^{-1}.$ Since $gT_ig^{-1}$ is a maximal torus containing $t'$ (from the fact that a regular semismiple element is contained in a unique maximal torus), we have $gT_ig^{-1}=T_i$, and so $g\in N_G(T_i)^F.$  Furthermore, by the proof of \cite[Proposition 3.7.3]{carter1985finite} we have that $C_G(t)$ is connected, and by \cite[Theorem 3.5.3]{carter1985finite} we have $C_G(t)^{\circ}=T_i.$ Since $C_G(t)=C_G(t)^{\circ}=T_i,$ it follows from Proposition \ref{prop:FstableTorifacts} that the $G^F$-class of $t$ in $T_i$ has size $|N_G(T_i)^F/C_G(t)^F|=|N_G(T_i)^F/T_i^F|=|C_W(w_i)|=c_i$. This gives the last statement. 
\end{proof}

 \subsection{Deligne-Lusztig characters}\label{subsect:Deligne-Lusztig}


When we consider irreducible characters of $G^F,$ the most important notion is a \emph{Deligne-Lusztig character} $R_T^G(\theta)$, where $T$ is a $F$-stable maximal torus of $G$ and $\theta\in \widehat{T^F}$. This is introduced by Deligne and Lusztig, and for definition, please see \cite{DL,geck2020character}. To estimate the number of zeroes in the character table of $G^F,$ we consider values of certain Deligne-Lusztig characters. 
In general, a Deligne-Lusztig character is a virtual character, however, there is a useful result about this character.
\begin{lem}\cite[Corollary 2.2.9]{geck2020character}A character $\theta\in \widehat{T[w]^F}$ is said to be \emph{in general position} when $w.\theta\neq \theta$ for any non-identity element $w$ in $ N_G(T[w])^F/T[w]^F$.
\begin{enumerate}
 \item If $\theta $ is in general position, the Deligne-Lusztig character $\epsilon_G\epsilon_T R_{T}^G(\theta)$ is an irreducible character of $G^F$. 
   \item If $\theta$ is not in general position, then the number of irreducible components of $R_T^G(\theta)$ is at most $|W|$.
 \end{enumerate}
 Here, $\epsilon_G:=(-1)^{r_G}$ (resp. $\epsilon_T:=(-1)^{r_T}$), where ${r_G}$ (resp. ${r_T}$) is the relative $F$-rank of a split maximal torus of $G$ (resp. the relative $F$-rank of $T$). 
 \end{lem}
 {For the definition of relative $F$-rank, please see \cite[Definition 2.2.11]{geck2020character}.}
Now, let us compare two irreducible Deligne-Lusztig characters.
\begin{lem}\label{lem:orthgonlaity}Let $T$ and $T'$ be two $F$-stable maximal tori of $G$, and $\theta\in\widehat{T^F}$ and $\theta' \in\widehat{T'^F}$ are in general position.
Let us assume that $T$ and $T'$ are not $G^F$-conjugate or $T=T'$ but $\theta$ and $\theta'$ are not $G^F$-conjugate. Then $R_T^G(\theta)$ and $R_{T'}^G(\theta')$ are non-isomorphic irreducible characters.
\end{lem}
\begin{proof}
From \cite[Theorem 6.8]{DL}, we have $$\langle R_T^G(\theta),R_{T'}^G(\theta')\rangle=|\{ w\in N(T,T')^F/T'^F\,|\, w\cdot\theta'=\theta\}|,$$ where $N(T,T')=\{ g\in G\,|\, Tg=gT'\}$. If $T$ and $T'$ are not $G^F$-conjugate, then $N(T,T')^F$ is empty. If $T=T'$ but $\theta$ and $\theta'$ are not $G^F$-conjugate, then $\{ w\in N(T,T')^F/T'^F\,|\, w\cdot\theta'=\theta\}$ is empty. So for both cases, we have $\langle R_T^G(\theta),R_{T'}^G(\theta')\rangle=0$, and this implies that $R_T^G(\theta)$ and $R_{T'}^G(\theta')$ are not isomorphic from the fact that both are irreducible (up to sign).
\end{proof}

These lemmas would play an important role in this paper. 

\subsubsection{Vanishing values of certain Deligne-Lusztig characters} Let us briefly explain why we consider Deligne-Lusztig characters. In general, for any irreducible character $\chi$ of $\widehat{G^F}$, it is hard to compute $\chi(g)$ for any $g\in G^F$. However, there is a well-known formula for the value of a Delign-Lusztig character, for example, \cite[Theorem 2.2.16]{geck2020character}. With this formula, we have the following useful result.

\begin{prop}\label{prop:vanishingvalue}
Let $T$ and $T'$ be a $F$-stable maximal tori of $G$ that are not $G^F$-conjugate, and consider $\theta\ (\in \widehat{T^F})$ is in general position, and $t\ (\in T'^F)$ is regular. Then $R_{T}^G(\theta)(t)=0$.
\end{prop}
\begin{proof}
From the formula \cite[Theorem 2.2.16]{geck2020character}, if $t$ is not conjugate in $G^F$ to any element of $T^F$, then the value is zero (cf. \cite[Example 2.2.17 (a)]{geck2020character}). Note that a regular semisimple element is in a unique maximal torus, so $t$ is not conjugate to any element of $T^F$ since $T$ and $T'$ are not $G^F$-conjugate. 
\end{proof}

\subsubsection{The number of in general position characters}\label{sss:generalregularelementsrelation} In this subsection, we consider how to count the number of in general position characters in $\widehat{T^F}$. 

Let us recall that there is a canonical bijection between $G^F$-orbits of pairs $(T,\theta)$ (where $T$ is a $F$-stable maximal torus of $G$ and $\theta \in \widehat{T^F}$) and $\chG^F$-orbits of pairs $(\check{T},s)$ (where $\check{T}$ is a $F$-stable maximal torus of $\chG$ and $s\in \check{T}^F$). This bijection is well-explained in \cite[Corollary 2.5.14]{geck2020character}. Briefly, if $(T,\theta)$ and $(\check{T},s)$ correspond and $T$ is twisted by $w\in W$ of a split maximal torus of $G$, then $\check{T}$ is twisted by $w^{-1}$ of a split maximal torus of $\chG$. {Note that in \cite{geck2020character}, they use the notation $w^*$ under an isomorphism $W(G)\rightarrow W(\chG)$. However, in our case, $w^*=w$ since the isomorphism $W(G)\rightarrow W(\chG)$ is the trivial in our paper. This is because this isomorphism is induced from the trivial map $X\rightarrow X$ with the fact that the root data of $G$ and $\chG$ are $(X,\Phi,Y,\Phi^\vee)$ and $(Y,\Phi^\vee,X,\Phi)$, cf. \cite[Proposition 4.2.3]{carter1985finite}.}
Furthermore, the construction of $s=s_\theta$ is written in \cite[\S2]{geck2020character}, but we leave the construction in this paper.
Now, if the pairs  $(T,\theta)$ and  $(\check{T},s)$  correspond to each other in this way, we write $R_{\check{T}}^G(s)=R_T^G(\theta)$. 
Let us introduce the following subset of $\widehat{G^F}$, which is called a {rational series of characters} of $G^F$, or {Lusztig series of characters}.
\begin{defe} \cite[Definition 2.6.1]{geck2020character}
Let $s\in \chG^F$ be a semisimple element. Then we define $\mathcal{E}(G^F,s)$ to be the set of $\chi \in \widehat{G^F}$ such that $\langle R_{\check{T}}^G(s),\chi \rangle\neq 0$ for some $F$-stable maximal torus $\check{T}\subset \chG$ with $s\in \check{T}$.
\end{defe} 

Then we have the following decomposition:
\begin{thm}\label{thm:bijectionlusztigjordan}
 \cite[Theorem 2.6.2]{geck2020character}
If $s_1,s_2\in \chG^F$ are semisimple and are $\chG^F$-conjugate, then $\mathcal{E}(G^F,s_1)=\mathcal{E}(G^F,s_2)$. Furthermore, 
\[
\widehat{G^F}=\underset{s}{\bigsqcup}\ \mathcal{E}(G^F,s),
\]
where $s$ runs over a set of representatives of the conjugacy classes of semisimple elements in $\chG^F$.
\end{thm}When the centre of $G$ is connected, we also have the following useful result:
\begin{thm}\label{thm:bijectiondimensionformula} \cite[Theorem 2.6.4]{geck2020character}
Let us assume that $Z(G)$ is connected. Then for a semisimple element $s\in \chG^F,$ there is a bijection
\[
\mathcal{E}(G^F,s)\overset{1-1}{\longleftrightarrow} Uch(C_{\chG}(s)^F),\quad \chi \longleftrightarrow \chi_u
\]
where $Uch(H^F)$ is the set of unipotent characters of $H^F.$ Furthermore, under the correspondence, we have that $\chi(1)=|\chG^F:C_{\chG}(s)^F|_{p'}\chi_u(1)$.
\end{thm}

\noindent With these results, let us compute the number of characters of in general position of $\widehat{T[w]^F}$.
\begin{prop}\label{prop:numberofingeneralpositionandregular}
The number of in general position characters of $\widehat{T[w]^F}$ (up to $G^F$-action) and the number of regular elements in $\check{T}[w]^F$ (up to $\chG^F$-action) are the same. Furthermore, the number of characters $\theta\in \widehat{T[w_i]^F}$ of in general position (up to $G^F$-action) is $\frac{\check{f_i}}{c_i}$, where $\check{f_i}$ is the number of regular elements in $\check{T}[w_i^{-1}]^F$.
\end{prop}

\begin{proof}

Let us assume that $(T[w],\theta)$ and $(\check{T}[w^{-1}],s)$ correspond. Then we need to show that $\theta$ is in general position if and only if $s$ is regular.

 If $\theta$ is in general position, we know that $\epsilon_G\epsilon_{T[w]}R_{T[w]}^G(\theta)$ is an irreducible character of $G^F$, and so we have that $\epsilon_G\epsilon_{T[w]}R_{T[w]}^G(\theta)\in \mathcal{E}(G^F,s)$. Let us assume that $s$ is not regular. Then from the dimension formula in Theorem \ref{thm:bijectiondimensionformula}, we have
\[
\epsilon_G\epsilon_{T[w]}R_{T[w]}^G(\theta)(1)=|\chG^F:C_{\chG}(s)^F|_{p'}\chi_u(1)
\]for a unipotent character $\chi_u\in Uch(C_{\chG}(s)^F)$. Recall that we have \[\epsilon_G\epsilon_{T[w]}R_{T[w]}^G(\theta)(1)=|G^F:T[w]^F|_{p'}.\] Then if $\chi_u$ is not the trivial character, there is a contradiction since $p \mid \chi_u(1)$ from Lemma \ref{lem:pdividethedimuni}. Then this implies that $\chi_u=1$, and we have 
\[
|G^F:T[w]^F|_{p'}=|\chG^F:C_{\chG}(s)^F|_{p'}\Rightarrow |T[w]^F|_{p'}=|C_{\chG}(s)^F|_{p'}.
\]
From our assumption on $G$, $C_{\chG}(s)$ is connected, so the order formula of $|C_{\chG}(s)^F|$ (cf. \cite[Theorem 1.6.7]{geck2020character}) implies that the degree of $|C_{\chG}(s)^F|_{p'}$ is at least $r+1$ (since we assumed that $s$ is not regular). However, it is obvious that the degree of $|T[w]^F|_{p'}$ is $r$, we get that $|T[w]^F|_{p'}\neq |C_{\chG}(s)^F|_{p'}$, which is a contradiction. Therefore, $s$ needs to be regular.

Conversely, let us assume that $s$ is regular, and then we have $|Uch(C_{\chG}(s)^F)|=1$ since $C_{\chG}(s)$ is a torus. This implies that $\mathcal{E}(G^F,s)$ is also a single set, and then the corresponding $R_{T[w]}^G(\theta)$ is irreducible, since irreducible components of $R_{T[w]}^G(\theta)$ are in $\mathcal{E}(G^F,s)$. Therefore, $\theta$ is in general position.

For the last statement, we need to count the number of regular elements in $\check{T}[w]^F$ (up to $\chG^F$-action). Note that $w$ and $w^{-1}$ are in the same conjugacy class (cf. \cite[\S9, Corollary]{carter1972conjugacy}), so we have that $T[w]^F\simeq T[w^{-1}]^F$. This means that $c_i=|C_W(w_i)|=|C_W(w_i^{-1})|$. Then, since the derived subgroup of $\chG$ is simply connected, we can prove this by applying the result in Proposition \ref{prop:properties-counts-of-regular-semisimple}. 
\end{proof}

\begin{lem}\label{lem:pdividethedimuni}
Let us assume that the size $q=p^e=|k^F|$ for $\mathrm{char}\ k=p$ satisfies conditions $\mathrm{(S1)}$--$\mathrm{(S3)}$ in Theorem \ref{thm:main}. Then for any non-trivial unipotent character $\rho$ of $G^F$, we have that $p\mid \rho(1)$.
\end{lem}

\begin{proof}
From \cite[\S3.3]{geck2018first}, $\rho(1)$ has of the form $\frac{1}{n_\rho}(q^{A_\rho}\pm\cdots \pm q^{a_\rho})$ (with $n_\rho \,|\, |W|$), and $a_\rho>0$ if and only if $\rho$ is non-trivial from \cite[Proposition 4.5.9]{geck2020character}. It is easy to see that if $p$ is larger than the size of Weyl group, then $p \mid \rho(1)$ is obvious since $n_\rho \mid |W|$.

Now, let us show that $p\mid \rho(1)$ in our assumption. From \cite[Chapter VI, \S2.4]{bourbaki}, $|W|=l!\cdot n_1n_2\cdots n_l\cdot f$, where $l$ is the semisimple rank of $G$, $n_i$ the coefficients of the highest root, and $f=|\Lambda/ \langle \Phi(G)\rangle |$ for the weight lattice $\Lambda$ of $G$. Note that  $q \nmid n_i$ for all $i$ from the definition of very good prime, and $q\nmid f$ from  \cite[Lemma 2.10 (b)]{herpel}. Then from assumption (S3), we can see that $p\mid \frac{q^{a_\rho}}{|W|}$, so we can conclude that $p\mid \rho(1).$
\end{proof}

\subsection{Counting lemma}
 
  In this subsection, we give  the size of $\widehat{G^F}$ and $|[G_{rss}^{F}]|$, where $[G_{rss}^{F}]$ denotes the set of regular semisimple conjugacy classes of $G^F$.  

\begin{prop}\label{prop:irresize}
	The sizes $|\widehat{G^F}|$ and $|[G_{rss}^F]|$ are monic polynomials of degree $r$.
\end{prop}
\begin{proof}
	The statement about the size of $|\widehat{G^F}|$ follows from \cite[\S5.2]{BK23} easily since the number of distinct irreducible characters is the same with the number of conjugacy classes in $G^F$. About the size $|[G_{rss}^{F}]|$, in \cite[\S5.2]{BK23}, they proved the previous result about $|\widehat{G^F}|$ using only regular semisimple elements. So this proof induces that the size $|[G_{rss}^{F}]|$ is a monic polynomial with degree $r.$
\end{proof}

\section{Sparsity of character tables of finite reductive groups} \label{s:proofmain} 

The central idea in the proofs of Theorems \ref{thm:fixedgroup} and \ref{thm:main} is to focus on Deligne–Lusztig characters arising from torus characters in general position, and to evaluate these characters on regular elements of tori in distinct conjugacy classes, as established in Proposition \ref{prop:vanishingvalue}.
This allows us to replace our initial limit by one that depends essentially only on the sizes of conjugacy classes in the Weyl group.


\subsection{A lower bound}\label{ss:lowerbound}

The first step of the proof our results based on the taking a lower bound of the following term:
\begin{equation*}
\frac{|\{(\chi,O) \in\widehat{G^{F}}\times [G^{F}]\,|\, \chi(O)=0\}|}{|\widehat{G^{F}}|\times |[G^{F}]|}\end{equation*}
for $G^F$ using observations in Section \ref{subsect:Deligne-Lusztig}.

\begin{prop}\label{prop:step1-genpos-semisimple-only}
	Let $G$ be a group satisfying the assumptions in Section \ref{subs:assumptions+notation}. Then we have
	  		\begin{equation}\label{eq:lowerbound}
		\frac{|\{(\chi,O) \in \widehat{G^{F}}\times[G^{F}]\mid \, \chi(O)=0\}|}{|\widehat{G^{F}}|\times |[G^{F}]|}\geq 
		\frac{|[G_{rss}^F]|\cdot |[\check{G}_{rss}^F]|}{|[G^{F}]|^2}-\frac{(q+1)^{2r}}{q^{2l}|Z(G)^F|^2}\sum_{i=1}^{\tau }\frac{1}{c_i^2}.
	\end{equation}
\end{prop}
Recall that here $\tau$ denotes the number of conjugacy classes in the Weyl group, $c_i$ the sizes of the centraliser subgroup of $w_i$ in $W$, $Z(G)$ the centre of $G$, and $r$ and $l$ denote the rank of $G$ and $[G,G],$ respectively.

\begin{proof}
	Recall that $T_i$ ($1\leq i\leq \tau$) are representatives of a conjugacy class of maximal tori in $G^F$, cf. Proposition \ref{prop:properties-counts-of-regular-semisimple}. 
	By Proposition \ref{prop:vanishingvalue}, if $ \theta\in \widehat{T_i^F}$ is an element in general position, $t\in T_j$ regular with $i\neq j$, then the Deligne-Lusztig character $R_{T_i}^G(\theta)\in \widehat{G^{F}}$ vanishes on the conjugacy class $[t]$ of $t.$ Using $|\widehat{G^{F}}|= |[G^{F}]|$, we then have
	\begin{equation}\label{eq:lowerbound_content}
	\frac{|\{(\chi,O) \in \widehat{G^{F}}\times[G^{F}]\,|\, \chi(O)=0\}|}{|\widehat{G^{F}}|\times |[G^{F}]|}\geq 
	\frac{\left|\left\{(R_{T_i}^G(\theta), [t]) \in \widehat{G^{F}}\times [G^{F}]\,\left|\, \substack{1\leq i,j\leq \tau,\ i\neq j,\\   \theta\in \widehat{T_i^F}\text{in general position,}\\ t\in T_j^F\text{ regular }} \right.\right\}\right|}{|[G^{F}]|^2}
	\end{equation}
	
    By Proposition \ref{prop:numberofingeneralpositionandregular}, the number of classes of regular elements in $\check{T_i}^F$ (equivalently, the number of classes of characters $ \theta\in \widehat{T_i^F}$ in general position) is $\displaystyle \frac{\check{f_i}}{c_i}.$ This gives a total of $\displaystyle |[\check{G}_{rss}^F]|=\sum_{i=i}^{\tau }\frac{\check{f}_i}{c_i}$ classes of regular semisimple elements in $\check{G}^F.$
	Similarly, by Proposition \ref{prop:properties-counts-of-regular-semisimple}, the number of conjugacy classes of regular semisimple elements of $G^F$ is $\displaystyle |[G_{rss}^F]|=\sum_{i=i}^{\tau }\frac{f_i}{c_i}$ with $\displaystyle \frac{f_i}{c_i}$ falling into each $T_i$ ($1\leq i\leq \tau$). 
Therefore the right-hand of equation \eqref{eq:lowerbound_content} is 
\[	\frac{|\{(\chi,O) \in\widehat{G^{F}}\times[G^{F}]\,|\, \chi(O)=0\}|}{|\widehat{G^{F}}|\times |[G^{F}]|}\geq 
	\frac{1}{|[G^{F}]|^2}\sum_{i=i}^{\tau }\left(\frac{\check{f_i}}{c_i}\sum_{\substack{1\leq j\leq \tau\\j\neq i}}\frac{f_j}{c_j}\right), 
\]	where the right-hand side can be rewritten as 
		\begin{equation}\label{eq:lowerbound_counts_summed}
		\frac{|\{(\chi,O) \in \widehat{G^{F}}\times[G^{F}]\,|\, \chi(O)=0\}|}{|\widehat{G^{F}}|\times |[G^{F}]|}\geq 	\frac{|[G_{rss}^F]|\cdot |[\check{G}_{rss}^F]|}{|[G^{F}]|^2}-\frac{1}{|[G^{F}]|^2}\sum_{i=1}^{\tau }\frac{f_i\check{f_i}}{c_i^2}.
	\end{equation}
	
To complete the proof of this proposition, it suffices to check that $\displaystyle \frac{f_i\check{f_i}}{|[G^{F}]|^2}\leq \frac{(q+1)^{2r}}{q^{2l}|Z^F|^2}$ for every $1\leq i\leq \tau. $ First, by \cite[Theorem 3.7.6]{carter1985finite} the number of conjugacy classes of semisimple elements in $G^F$ is $q^l|Z^F|,$ so we have $|[G^F]|\geq q^l|Z^F|.$ Second, we have $f_i\leq |T_i^F|$ and $\check{f}_i\leq |\check{T}_i|$ and by \cite[Lemma 2.3.11 equation(a)]{geck2020character}, and we derive $|T_i^F|\leq (q+1)^r$ and $|\check{T}_i|\leq (q+1)^r.$ Thus, the right-hand side of equation \eqref{eq:lowerbound_counts_summed} is bounded from below by the right-hand side of equation \eqref{eq:lowerbound}. This completes the proof. 
\end{proof}

\subsection{Proof of Theorem \ref{thm:fixedgroup}}\label{ss:proofmain2} 
The proof of Theorem \ref{thm:fixedgroup} is almost immediate from Proposition \ref{prop:step1-genpos-semisimple-only}. Recall the setting of Theorem \ref{thm:fixedgroup}. We have a sequence of groups $G_n$ that all have essentially the same structure, and a Frobenius map $F_n$ with a fixed field of order $q_n.$ We have $\lim_{n\rightarrow \infty }q_n=\infty .$ For each $G_n$, the corresponding Weyl group $W$ is the same, and it has $\tau $ conjugacy classes of sizes $c_i$ for $1\leq i\leq \tau .$

By Proposition \ref{prop:step1-genpos-semisimple-only} we have the following: 
\[
\lim_{n\rightarrow \infty} \frac{|\{(\chi,O) \in \widehat{G_n^{F_n}}\times[G_n^{F_n}]\mid \, \chi(O)=0\}|}{|\widehat{G_n^{F_n}}|\times |[G_n^{F_n}]|}\geq \lim_{n\rightarrow \infty} \left(\frac{|[G_{n,rss}^{F_n}]|\cdot |[\check{G}_{n,rss}^{F_n}]|}{|[G_n^{F_n}]|^2}-\frac{(q_n+1)^{2r}}{q_n^{2l}|Z(G_n)^{F_n}|^2}\sum_{i=1}^{\tau }\frac{1}{c_i^2}\right)
\]
Recall from Proposition \ref{prop:irresize} that here $|[G_{n,rss}^{F_n}]|,$ $|[\check{G}_{n,rss}^{F_n}]|,$ and $ |[G_n^{F_n}]|$ are all monic polynomials of $q_n$ of degree $r,$ i.e. the rank of $G_n;$ note that the rank is independent of $n.$ In addition, the size of the centre $|Z(G_n)^{F_n}|$ is a monic polynomial in $q_n$, whose form is independent of $n$ (from the split assumption), 
and whose degree is $r-l.$

Let us recall that  $|[G_{n,rss}^{F_n}]|,$ $|[\check{G}_{n,rss}^{F_n}]|,$ and $ |[G_n^{F_n}]|$ are polynomial on residue classes (cf. \cite{BK23}), which means that the formulae of  $|[G_{n,rss}^{F_n}]|,$ $|[\check{G}_{n,rss}^{F_n}]|,$ and $ |[G_n^{F_n}]|$ depend on $i$ of  $q_n \equiv i \pmod{d}$ for a fixed integer $d$. Note that for each $i$, the size $|[G_{n,rss}^{F_n}]|,$ $|[\check{G}_{n,rss}^{F_n}]|,$ and $ |[G_n^{F_n}]|$ are monic with degree $r$. The exact polynomials depend on the residue $i$ ($0\leq i\leq d-1$), but in each case, the quotient $\frac{|[G_{n,rss}^{F_n}]|\cdot |[\check{G}_{n,rss}^{F_n}]|}{|[G_n^{F_n}]|^2}$ has a numerator and denominator of degree $2r.$ Each of these rational functions has limit $1$ as $n\rightarrow \infty .$ Since there are finitely many ($d$) of them, we can take a shared lower bound that is independent of $q$ even the residue $q \mod d.$ It follows that 
$$\lim _{n\rightarrow \infty }\frac{|[G_{n,rss}^{F_n}]|\cdot |[\check{G}_{n,rss}^{F_n}]|}{|[G_n^{F_n}]|^2}=1.$$
%
%
%
Then it follows that 
$$\lim_{n\rightarrow \infty} \left(\frac{|[G_{n,rss}^{F_n}]|\cdot |[\check{G}_{n,rss}^{F_n}]|}{|[G_n^{F_n}]|^2}-\frac{(q_n+1)^{2r}}{q_n^{2l}|Z(G_n)^{F_n}|^2}\sum_{i=1}^{\tau }\frac{1}{c_i^2}\right)=1-\sum_{i=1}^{\tau }\frac{1}{c_i^2},$$
which completes the proof of Theorem \ref{thm:fixedgroup}. 



\subsection{Probability that Weyl elements are conjugate}  
We have seen above that the sum 
$$\sum_{i=1}^{\tau} \frac{1}{c_i^2}=\sum_{i=1}^{\tau }\frac{1}{|C_{W}(w_{i})|^2},$$
plays a role in our computations. Recall that $\tau$ is the number of conjugacy classes in a Weyl group, $w_i$ are representatives, and $c_i$  are the sizes of these conjugacy classes ($1\leq i\leq \tau$). This is the probability that two elements of the Weyl group are conjugate to each other.

We will show that for simple groups of large rank, this probability is asymptotically zero. 
The main ingredient in the proof is the fact that this holds in the symmetric group $S_n$, as $n$ approaches infinity. This is proved by Blackburn, Britnell and Wildon in \cite[Theorem 1.4]{blackburn2012probability}. We shall use this fact in the proof of Theorem \ref{thm:main}. 

\begin{thm}\label{thm:cited_symmetric-group-result}
	\cite[Theorem 1.4]{blackburn2012probability}
	Let $\pi_1,\ldots ,\ \pi_{p(r)}$ be representatives of the different conjugacy classes in $S_r,$ the symmetric group on $n$ letters. (Recall that here $p(r)$ is the number of partitions of the positive integer $r.$) Then we have 
	$$\lim _{n\rightarrow \infty } 	\sum_{i=1}^{p(r)}\frac{1}{|C_{S_r}(\pi_i)|^2} = 0$$
	where $C_{S_r}(\pi_i)$ denotes the centraliser of $\pi_i$ in $S_r.$ More precisely, we have 
	$$\sum_{i=1}^{p(r)}\frac{1}{|C_{S_r}(\pi_i)|^2}\leq \frac{C}{r^2}$$
	for a fixed constant $C;$ in particular, $C<6.$
\end{thm}  
\begin{proof}
This is \cite[Theorem 1.4]{blackburn2012probability}; note  that the constant $C$ is chosen such that equality holds for $r=13.$ The estimate $C<5.49<6$ follows from \cite[Lemma 8.2 (v)]{blackburn2012probability}. 
\end{proof}

The rest of the proof of Theorem \ref{thm:main} relies on the observation that Weyl groups in types of large semisimple ranks are ``close'' to $S_n$ and a lemma from real analysis.

\begin{lem}\label{lem:S_n-enough}Let $G$ be simple of rank $r\geq 9.$ Then 
	$$\sum_{i=1}^{\tau} \frac{1}{c_i^2}=\sum_{i=1}^{\tau (W)}\frac{1}{|C_{W}(w_{i})|^2}\leq \frac{6}{r^2}.$$ 
\end{lem} 
\begin{proof}
	From the assumption on the rank, $G$ is of Cartan type $A_r,$ $B_r,C_r$ or $D_r$ ($r \geq 9$). In type $A_r,$ the Weyl group is isomorphic to $S_{r+1}$ and the statement follows from Theorem \ref{thm:cited_symmetric-group-result} above. In types $B_r,$ $C_r$ or $D_r$ the Weyl group $W$ is isomorphic to a semidirect product $N\rtimes S_r$ for some normal subgroup $N$ of $W$, cf. \cite{PR16}. Then from the definition of a semidirect product, $W/N \simeq S_r$, and \cite[Lemma 3.1]{blackburn2012probability} implies that  
	$\sum_{i=1}^{\tau (W)}\frac{1}{|C_{W}(w_{i})|^2}\leq \sum_{i=1}^{p(r)}\frac{1}{|C_{S_r}(\pi_i)|^2}$. The statement now again follows from Theorem \ref{thm:cited_symmetric-group-result} above.
\end{proof} 

\begin{rem}\label{rem:smallrankgroups}
The conditions of Lemma \ref{lem:S_n-enough} are set to avoid exceptional Cartan types. However, an estimate on these is not too difficult to give. In fact, a comparison between the structure of Weyl groups in \cite[12.2, Table 1]{humphreys2012introduction} and \cite[Theorem 1.3]{blackburn2012probability} quickly shows that the left-hand side is less than $\frac{1}{4}$ unless $G$ is a torus, or is of type $A_1,$ $A_1\times A_1$ or $A_3.$
\end{rem}

\subsection{Simple components}\label{ss:simple-components-estimates}

The proof of Theorem \ref{thm:main} will proceed by estimating the proportion of non-vanishing entries in the character table of a group by the same quantity for the group's simple components. We first explain the relationship of these components to the group, and then derive a more precise version of Proposition \ref{prop:step1-genpos-semisimple-only} for simple groups. 

\begin{prop}\label{prop:group-factorising}
Let $G_n,$ $F_n$ and $q_n$ be as in Theorem \ref{thm:main}. Then we may write 
\begin{equation}
\label{eq:factors-of-Gn}
G_n^{F_n}=H_{1,n}^{F_n}\times H_{2,n}^{F_n}\times\cdots \times H_{m_n,n}^{F_n}\times (Z(G_n)/K_n)^{F_n}
\end{equation}
where $1\leq m_n$ and $H_{i,n}$ $(1\leq i\leq m_n)$ are groups satisfying the conditions of Theorem \ref{thm:main} that are, in addition, simple; $K$ is a finite group. In addition, if $r(G)$ denotes the rank of a group $G,$ we have 
\begin{equation}\label{eq:relation-rank-factors}
	r(G_n)= r(Z(G_n))+\sum_{i=1}^{m_n} r(H_{i,n}).
\end{equation}
Furthermore, letting 
\[
P(G):= 	1-\frac{|\{(\chi,O) \in \widehat{G^{F}}\times[G^{F}]\mid \, \chi(O)=0\}|}{|\widehat{G^{F}}|\times |[G^{F}]|}
\]
the proportion of non-zero entries in the character table of a finite group $G^F,$ then we have
\begin{equation}\label{eq:relation-P(G)-factors}
	P(G_n )= \prod_{i=1}^{m_n} P(H_{i,n} ).
\end{equation}
\end{prop}
\begin{proof}
 
By \cite[p. 61]{geck2020character}, one has an isogeny $\phi:[G_n,G_n]\times Z(G_n)\rightarrow G_n$ given by $\phi(g,z)=gz;$ the kernel of this isogeny is the finite set $K_n:=\ker \phi =[G_n,G_n]\cap Z(G_n).$ Thus we have $G_n\simeq ([G_n,G_n]\times Z(G_n))/K_n \simeq [G_n,G_n]\times Z(G_n)/K_n$ whence $G_n^{F_n}\simeq [G_n,G_n]^{F_n}\times (Z(G_n)/K_n)^{F_n}.$
By \cite[\S 1.5.11, \S 1.5.14]{geck2020character}, we have $[G_n,G_n]=H_{1,n}\times \cdots \times  H_{m_n,n}$. 

Now, let us consider equation \eqref{eq:factors-of-Gn}, concerning the $F_n$-fixed points of these groups. For $G_n\simeq H_{1,n}\times \cdots \times  H_{m_n,n}\times (Z(G_n)/K_n)$, we have that $F_n$ acts on the set $\{H_{1,n},\ldots ,H_{m_n,n}\}$. Note that the following can occur: $F(H_{i,n})= H_{j,n} $ for some $i \neq j$.
By \cite[Corollary 1.5.16]{geck2020character}, we have: 
\begin{equation}\label{eq:factor-orbit}
[G_{n}^{F_n}, G_{n}^{F_n}]\simeq \prod_{i \in I} H_{i,n}^{F_n^{r_i}},
\end{equation}
where $I \subset \{ 1,2, \ldots , m_n\}$ and $r_i$ is the length of $F$-orbits on $H_{1,n}, \ldots ,  H_{m_n,n}$. 
We claim that in our case each orbit is a singleton, i.e. $r_i=1$ for all $i$, and we will prove this using point-counts. This will imply equation \eqref{eq:factors-of-Gn}. 

 To compare the point-counts on the two sides of equation \eqref{eq:factor-orbit}, we use \cite[Theorem 1.6.7]{geck2020character}. Note that since we assumed $G_n$ is untwisted, we also have that each $H_{i,n}$ is also untwisted. (Since $W(G_n)=\prod_{i=1}^{m_n} W(H_{i,n})$.) Then the automorphism  induced by $F_n$ on the character lattice of $G$ (as well as on the character lattice of each of the $H_{i,n}$) is the identity automorphism, cf. \cite[Example 1.4.21]{geck2020character}. Furthermore, under the assumption that $G_n$ is split, each $H_{i,n}$ is also split. Thus we have that 
\[
|[G_{n}^{F_n}, G_{n}^{F_n}]|=q_n^{|\Phi^+|} (q_n-1)^{l_n} \prod_{w\in W(G_n)} q_n^{\ell(w)} ,
\]
where $\ell$ is the length function of $W(G_n)$, $\Phi $ is the root system of $G_n$ and $l_n=r([G_n,G_n])$ is the semisimple rank of $G_n$. Similarly, we have for any $k \in \mathbb{Z}$:
\[
|H_{i,n}^{F_n^k}|=q_n^{k|\Phi_i^+|} (q_n^k-1)^{r(H_{i,n})} \prod_{w\in W(H_{i,n})} q_n^{k\ell(w)},
\]
where $\Phi_{i,n}$ is the root system of $H_{i,n}.$ 
Let us compare the power of $q-1$ dividing the size of the two sides of equation \eqref{eq:factor-orbit}. On the left-hand side, this power is exactly $l_n.$ 
On the right-hand side, we have $\displaystyle \sum_{i\in I} r(H_{i,n}).$ If $I$ is a proper subset of $ \{ 1,2, \ldots , m_n\}$ (equivalently, there is $i$ such that $r_i\geq 2$), then the powers of $q-1$ dividing the size of the two sides of equation \eqref{eq:factor-orbit} are different. Then this is a contradiction, so we have $r_i=1$ for all $i$. Note that this easily implies that equation \eqref{eq:relation-rank-factors} holds.
The fact that $H_{i,n}$ ($1\leq i\leq m_n$) satisfy the remainder of the conditions in Theorem \ref{thm:main} (any of which are simply connected and adjoint type) follows from the relationship between their root datum and the root datum of $G_n.$ 


Finally, we must prove equation \eqref{eq:relation-P(G)-factors}. An irreducible character $\chi $ of $G_n^{F_n}$ is a product of irreducible characters $\chi_1\cdots \chi _{m_n}\cdot \theta $ of the groups $H_{i,n}^{F_n}$ and a character $\theta $ of $(Z(G_n)/K_n)^{F_n}$ respectively. The conjugacy classes of $G_n^{F_n}$ are in bijective correspondence with $m_n$-tuples of conjugacy classes of the $H_{i,n}^{F_n}$ and if for $G_n^{F_n}\ni g=g_1g_2\cdots g_{m_n}$ in the direct product equation \eqref{eq:factors-of-Gn} then $\chi(g)\neq 0$ if and only if $\chi_i(g_i)\neq 0$ for all $1\leq i\leq m_n.$ This implies equation \eqref{eq:relation-P(G)-factors}. 
\end{proof}

Proposition \ref{prop:group-factorising} shows that the proportion of nonzero entries in the character table of $G^F$ can be related to the same for the simple components. To use this in the proof of Theorem \ref{thm:main}, we will phrase a careful estimate for simple groups based on a special case of Proposition \ref{prop:step1-genpos-semisimple-only}. In order to do so, we will need a few results on the quantities appearing in that proposition.

\begin{lem}\label{lem:lowerbound}
Let $H$ be a simple group satisfying the conditions of Theorem \ref{thm:main}, and let $F$ and $q$ be as in that Theorem. 
Then we have
\begin{equation}\label{eq:simple-firstterm-lowerbound}
	\frac{|[H_{rss}^F]|\cdot |[\check{H}_{rss}^F]|}{|[H^{F}]|^2}\geq \frac{((q-1)^r-3(q-1)^{r-1}-2(q-1)^{r-2})^2}{(q^r+40q^{r-1})^2}.
\end{equation}
\end{lem}
\begin{proof}
It follows from \cite[Theorem 1.1]{guralnick2001p} that 
\begin{equation}\label{eq:guralnick-lubeck-result}
\frac{|H_{rss}^F|}{|H^F|}>1-\frac{3}{q-1}-\frac{2}{(q-1)^2}=\frac{(q-1)^r-3(q-1)^{r-1}-2(q-1)^{r-2}}{(q-1)^r}.
\end{equation}
Observe that for a regular, semisimple element $h\in H$ we have that $C_H(h)$ is a maximal torus. It follows from \cite[Lemma 2.3.11 equation(a)]{geck2020character} that $|C_H(h)^F|\geq (q-1)^r.$ If $[h]$ denotes the $H^F$ conjugacy class of $h,$ we have $|[h]|\cdot |C_H(h)^F|=|H^F|.$  Thus we have 
\begin{equation}\label{eq:estim-portion-semisimple-q}
\frac{|H_{rss}^F|}{|H^F|}=   \frac{\sum_{[h]\in [H_{rss}^F]} |[h]|}{|H^F|}=\sum_{[h]\in [H_{rss}^F]}\frac{1}{|C_H(h)^F|}\leq \sum_{[h]\in [H_{rss}^F]}  (q-1)^{-r}=|[H_{rss}^F]|\cdot (q-1)^{-r}.
\end{equation}
Comparing equation \eqref{eq:guralnick-lubeck-result} and equation \eqref{eq:estim-portion-semisimple-q}, we get 
\begin{equation}\label{eq:estimate-numerator-H}
|[H_{rss}^F]|>(q-1)^r-3(q-1)^{r-1}-2(q-1)^{r-2}.
\end{equation}
Since $H$ is a simply connected, simple algebraic group of rank $r$, it follows from \cite[Corollary 5.1]{fulman2012bounds} that 
\begin{equation}\label{eq:H-class-estimate-simple}
|[H^F]|\leq q^r+40q^{r-1}.
\end{equation}
Combining equation \eqref{eq:estimate-numerator-H} for $H$ and $\check{H}$ (which have the same rank) and equation \eqref{eq:H-class-estimate-simple} yields the bound in equation \eqref{eq:simple-firstterm-lowerbound} as desired. 
\end{proof}

We are now ready to phrase a special case of Proposition \ref{prop:step1-genpos-semisimple-only} for simple groups. 

\begin{cor}\label{cor:simple-group-lower-bound}
Let $H$ be a simple group satisfying the conditions of Proposition \ref{prop:step1-genpos-semisimple-only}. Let $F$ be as in Theorem \ref{thm:main}, and let other notation correspond to $H$ as before (e.g. $r=r(H)$).  Then we have 
		\begin{equation}\label{eq:bound-simplegroup}
	1-P(H)=\frac{|\{(\chi,O) \in \widehat{H^{F}}\times[G^{F}]\mid \, \chi(O)=0\}|}{|\widehat{H^{F}}|\times |[H^{F}]|}\geq 
	\frac{f_{1,r}(q)}{f_{2,r}(q)}-\frac{(q+1)^{2r}}{q^{2r}}\sum_{i=1}^{\tau }\frac{1}{c_i^2},
\end{equation}
where $f_1(x),\ f_2(x)\in \Z[x]$ are monic polynomials of degree $2r$ that do not depend on $H$ beyond dependence on $r.$ More explicitly, 
\[
\begin{split}
f_{1,r}(x)=& ((x-1)^r-3(x-1)^{r-1}-2(x-1)^{r-2})^2;\\
f_{2,r}(x)=& (x^r+40x^{r-1})^2.
\end{split}
\]
\end{cor}
\begin{proof}
Since $H$ is simple and adjoint, we have that $H=[H,H]$ and $Z(H)$ is trivial, so by Proposition \ref{prop:step1-genpos-semisimple-only}, we immediately have 
		\[
	1-P(H)\geq 
	\frac{|[H_{rss}^F]|\cdot |[\check{H}_{rss}^F]|}{|[H^{F}]|^2}-\frac{(q+1)^{2r}}{q^{2r}}\sum_{i=1}^{\tau }\frac{1}{c_i^2}.
\]
This corollary now follows from the lower bound on the first term from Lemma \ref{lem:lowerbound}. 
\end{proof}

\subsection{The proof of Theorem \ref{thm:main}}\label{ss:proof-main}

We are ready to prove Theorem \ref{thm:main}. 
Let $G_n,$ $F_n,$ $q_n$ and $l_n=r([G_n,G_n])$ be as in Theorem \ref{thm:main}, and let us use the notation of Proposition \ref{prop:group-factorising} so that we have 
$$G_n^{F_n}=H_{1,n}^{F_n}\times H_{2,n}^{F_n}\times\cdots \times H_{m_n,n}^{F_n}\times (Z(G_n)/K_n)^{F_n}$$
and 
$$	r(G_n)= r(Z(G_n))+\sum_{i=1}^{m_n} r(H_{i,n}),\text{ that is, }l_n=\sum_{i=1}^{m_n}r(H_{i,n}).$$
To prove Theorem \ref{thm:main}, we wish to show that 
\begin{equation}\label{eq:want-to-show-main-product}
\lim_{n\rightarrow \infty}P(G_n^{F_n})= \lim_{n\rightarrow \infty}\prod_{i=1}^{m_n} P(H_{i,n}^{F_n})=0.
\end{equation}
Recall that as part of the assumptions, we have $q_n\geq l_n,$ and $\displaystyle \lim_{n\rightarrow \infty}l_n=\infty. $ 
Note that equation \eqref{eq:want-to-show-main-product}, shown above in equation \eqref{eq:relation-P(G)-factors}, is a key to the proof. We explain the general idea first. Certainly $P(H_{i,n} )<1$ for every $1\leq i\leq m_n.$ 
When $r(H_{i,n})$ is large, the combination of Corollary \ref{cor:simple-group-lower-bound} and Lemma \ref{lem:S_n-enough} means that $P(H_{i,n} )$ is approximately the probability of two elements being conjugate in the symmetric group on $r(H_{i,n})$-letters. This asymptotically vanishes by the work of Blackburn, Britnell and Wildon cited in Theorem \ref{thm:cited_symmetric-group-result}. 
When $r(H_{i,n})$ is small, Corollary \ref{cor:simple-group-lower-bound} still gives us an upper bound on $P(H_{i,n} ).$ This is less strict but still sufficient to give an upper bound of $1-\varepsilon $ for some $\varepsilon>0$ that is independent of $H_{i,n},$ provided $q_n$ is large enough. 

Roughly speaking, as $\lim_{n\rightarrow \infty}l_n=\infty ,$ the group $G_n$ has either many factors $H_{i,n}$ of bounded rank, yielding $P(G_n)<(1-\varepsilon)^K$ for some large $K,$ or a factor $H_{i,n}$ of very large rank, yielding $P(G_n)<P(H_{i,n})$ close to zero. 

One technical difficulty is that while the bounds in Corollary \ref{cor:simple-group-lower-bound} do not depend on the type of the simple group $H,$ they do depend on the rank $r(H).$ This necessitates the somewhat techincal-looking lower bound on $q_n$ in the statement of Theorem \ref{thm:main}. 
We shall now make the above idea of the proof precise. 
The first two lemmas are of a purely technical nature; they help us understand the estimates in Corollary \ref{cor:simple-group-lower-bound} better. 

\begin{lem}\label{lem:technical-struggle-for-estimates-fixed-r}
Fix a positive integer $R_0$ and $\varepsilon_0>0.$ Then there exists a natural number $q_0\in \N$ such that for $q\geq q_0$ and any $1\leq r\leq R_0$, we have 
\begin{equation}\label{eq:bound-estimates-technical-fixed-r}
	\frac{(q +1)^{2r}}{q ^{2r}}<1+\varepsilon_0 \text{ and }1-\frac{f_{1,r}(q )}{f_{2,r}(q )}<\varepsilon_0.
\end{equation}
\end{lem}
\begin{proof}
Both $\displaystyle \frac{(q+1)^{2r}}{q^{2r}}$ and $\displaystyle \frac{f_{1,r}(q )}{f_{2,r}(q )}$ are the quotients of monic polynomials of degree $2r;$ thus for a fixed $r$ we have 
$\displaystyle \lim_{n\rightarrow \infty} \frac{(q +1)^{2r}}{q ^{2r}}=\lim_{n\rightarrow \infty} \frac{f_{1,r}(q )}{f_{2,r}(q )}=1.$ This implies that $q_{0,r}\in \N$ can be chosen so that the inequalities in equation \eqref{eq:bound-estimates-technical-fixed-r} are satisfied for a particular $r.$ Setting $q_0$ as the maximum of the finite set $\{q_{0,r}\mid 1\leq r\leq R_0\}$ satisfies the requirements of this lemma. 
\end{proof}

The second technical lemma involves slightly more complicated computations. It helps us understand the estimate in Corollary \ref{cor:simple-group-lower-bound} when the rank $r=r(H)$ is large, and therefore the convergence of the coefficients $\displaystyle \frac{(q+1)^{2r}}{q^{2r}}$ and $\displaystyle \frac{f_{1,r}(q )}{f_{2,r}(q )}$ to $1$ is potentially slower.

\begin{lem}\label{lem:technical-struggle-for-estimates}
Let $f:\Z_{>0}\rightarrow {\mathbb{R}}$ be a function on positive integers so that $\displaystyle \lim _{r\rightarrow \infty } f(r)=\infty .$ Let $\varepsilon _0>0.$ Then there is an $r_0\in \N$ such that if $r\geq r_0$ and $q\geq r\cdot f(r)$ then we have 
\begin{equation}\label{eq:bound-estimates-technical}
\frac{(q+1)^{2r}}{q ^{2r}}<1+\varepsilon_0 \text{ and }1-\frac{f_{1,r}(q)}{f_{2,r}(q)}<\varepsilon_0.
\end{equation}
\end{lem}
\begin{proof}
First note that since $q\geq r\cdot f(r)$ we have 
$$\frac{(q +1)^{2r }}{q ^{2r }}=\left(1+\frac{1}{q }\right)^{2r }\leq \left(1+\frac{1}{r \cdot f(r )}\right)^{2r }=\left(\left(1+\frac{1}{r \cdot f(r )}\right)^{r \cdot f(r )}\right)^{\frac{2}{f(r )}}.$$

Recall that $\displaystyle \lim_{x\rightarrow \infty }\left(1+\frac{1}{x}\right)^x=e$ (Euler's constant). Since $\displaystyle\lim _{r\rightarrow \infty } f(r )=\infty ,$ there exists an $r_1\in \N$ such that for every $r\geq r_1$ we have 
$$\left(1+\frac{1}{r \cdot f(r )}\right)^{r \cdot f(r )}<e^{\frac{3}{2}}\quad \text{ and }\quad \frac{3}{\ln(1+\varepsilon_0)}<f(r ).$$
Putting the above inequalities together yields that for $r\geq r_1$, the first inequality in equation \eqref{eq:bound-estimates-technical} is satisfied.

For the second inequality in equation \eqref{eq:bound-estimates-technical}, we start with the straightforward observation that 
$$\frac{f_{1,r}(x)}{f_{2,r}(x)}=\left(1-\frac{1}{x}\right)^{2r-4}\cdot \frac{x^4-10x^3+29x^2-20x+4}{x^4+80x^3+1600x^2}=\left(1-\frac{1}{x}\right)^{2r-4}\cdot g(x)$$
Here $g(x)$ is a rational function that is the quotient of two quartic monic polynomials, in particular, $\displaystyle \lim_{x\rightarrow \infty}g(x)=1.$ Furthermore, for $q \geq r f(r )$ as in the lemma, we have 
$$\left(1-\frac{1}{q }\right)^{2r -4}\geq \left(1-\frac{1}{r f(r )}\right)^{2r -4}=\left( \left(1-\frac{1}{r f(r )}\right)^{r f(r )}\right)^{\frac{2}{f(r )}-\frac{4}{r f(r )}}.$$
Recall that $\displaystyle \lim_{x\rightarrow \infty }\left(1-\frac{1}{x}\right)^x=e^{-1};$ in particular there is an $r_2\in \N$ such that for $r\geq r_2$ we have 
$$1>\left(1-\frac{1}{r f(r )}\right)^{r f(r )}>e^{-2}\Rightarrow 1>\left(1-\frac{1}{q }\right)^{2r -4}\geq \left(1-\frac{1}{r f(r )}\right)^{2r -4}>\left(e^{-1}\right)^{\frac{4}{f(r )}-\frac{8}{r f(r )}}.$$ Observe that we may assume that $\varepsilon_0<1.$
Since $\displaystyle \lim_{r\rightarrow \infty}f(r )=\infty ,$ there is an $r_3\in \N $ so that for $r\geq r_3$ we have $\displaystyle 1>\left(e^{-1}\right)^{\frac{4}{f(r )}-\frac{8}{r f(r )}}>1-\varepsilon_0^2.$ Furthermore, since $\lim_{x\rightarrow \infty}g(x)=1$ there is an $r_4\in \N $ so that for $q\geq rf(r)$ such that $r\geq r_4$ we have $g(q )>\frac{1-\varepsilon_0}{1-\varepsilon_0^2}.$ Putting all this together yields that for $r\geq \max\{ r_2,r_3,r_4\}$  and $q\geq rf(r)$ we have 
$$1-\frac{f_{1,r }(q )}{f_{2,r }(q )}<1-\left(1-\frac{1}{q }\right)^{2r -4}\cdot g(q )<1-(1-\varepsilon_0^2)\cdot \frac{1-\varepsilon_0}{1-\varepsilon_0^2}=\varepsilon_0.$$
Now, setting $r_0=\max\{r_1, r_2,r_3,r_4\}$ the proof of this lemma is complete. 
\end{proof}
To estimate the contribution of low-rank simple components, we use the following.

\begin{lem}\label{lem:smallrank-1-epsilon1}
Let $R_0$ be a fixed number. Then there are numbers $1>\varepsilon_1>0$ and $q_{\varepsilon_1}$ such that if $H$ is a simple (non-abelian) group satisfying the conditions of Theorem \ref{thm:main} (defined over a field $k$) and additionally $r(H)\leq R_0,$ and $F$ is such that $|k^F|=q\geq q_{\varepsilon_1}$ then we have
\[
P(H^{F})<1-\varepsilon_1. 
\]
\end{lem}
\begin{proof}
The proof uses Corollary \ref{cor:simple-group-lower-bound} to give an upper bound on $P(H^{F}).$ Parts of the upper bound will be further estimated using Lemma \ref{lem:S_n-enough} (for $r(H)\geq 9$), Remark \ref{rem:smallrankgroups} (for $9\leq r(H)\leq R_0$) and Lemma \ref{lem:technical-struggle-for-estimates-fixed-r} (used for all $1\leq r(H)\leq R_0$). We will use equation \eqref{eq:bound-simplegroup}, so let us consider the term $\sum_{i=1}^{\tau}\frac{1}{c_i^2} $.

First, note that there are finitely many possible types of $H$ with $r(H)\leq 8,$ and $\displaystyle \sum_{i=1}^{\tau}\frac{1}{c_i^2}<1$ for any of these. (Recall the notation of Corollary \ref{cor:simple-group-lower-bound}, e.g. the $c_i$ refer to the sizes of the centralisers in the Weyl group of $H,$ and see also Remark \ref{rem:smallrankgroups}.)

Choose $\varepsilon_1 $ such that $\displaystyle \sum_{i=1}^{\tau}\frac{1}{c_i^2}<1-3\varepsilon_1$ for each possible simple type of $H$ with $r(H)\leq 8$ and additionally $\varepsilon_1<\frac{25}{81}.$ 
The latter condition ensures $\frac{2}{27}<1-3\varepsilon_1,$ thus by Lemma \ref{lem:S_n-enough} we have that if $r(H)\geq 9$ then  $\displaystyle \sum_{i=1}^{\tau}\frac{1}{c_i^2}\leq \frac{6}{r^2}\leq \frac{2}{27}<1-3\varepsilon_1.$ 

With $\varepsilon_1$ now chosen, by Lemma \ref{lem:technical-struggle-for-estimates-fixed-r} we have a $q_{\varepsilon_1}$ such that if $q\geq q_{\varepsilon_1}$ and $1\leq r\leq R_0$ we have 
$$\frac{(q+1)^{2r}}{q^{2r}}<\frac{1-2\varepsilon_1}{1-3\varepsilon_1}\text{ and }\left(1-\frac{f_{1,r}(q)}{f_{2,r}(q)}\right)<\varepsilon_1 .$$

Then by Corollary \ref{cor:simple-group-lower-bound} we have that if $H,$ $F,$ $q$ are as in the statement of this lemma, then
\[
P(H^F)\leq  \left(1-\frac{f_{1,r(H)}(q)}{f_{2,r(H)}(q)}\right) +\frac{(q+1)^{2r(H)}}{q^{2r(H)}}\sum_{i=1}^{\tau }\frac{1}{c_i^2}<\varepsilon_1+\frac{1-2\varepsilon_1}{1-3\varepsilon_1}\cdot (1-3\varepsilon_1)=1-\varepsilon_1.
\]
This completes the proof. 
\end{proof}


\subsubsection{The proof of Theorem \ref{thm:main} }
Using the above auxiliary lemmas, we can now make our plan to prove Theorem \ref{thm:main} precise.

Consider the sequence of groups $\{G_n\}_{n=1}^\infty$ satisfying the conditions of Theorem \ref{thm:main} with a function $f:\Z_{>0}\rightarrow \R$, and recall that for $l_n=[G_n,G_n]$, we have $\displaystyle \lim_{n\rightarrow \infty}l_n=\infty $ and $\displaystyle \lim_{l\rightarrow \infty}f(l)=\infty ,$ as well as $q_n\geq l_nf(l_n).$ Let us fix a number $\varepsilon>0.$ 
It follows from Lemma \ref{lem:technical-struggle-for-estimates} that there is a $r_0\in \N$ such that for $r\geq r_0$ and $q\geq r\cdot f(r)$ we have 
\[
\frac{6}{r^2}<\frac{\varepsilon}{4},\quad \left(1-\frac{f_{1,r}(q)}{f_{2,r}(q)}\right)<\frac{\varepsilon}{2}\quad \text{and}\quad \frac{(q+1)^{2r}}{q^{2r}}<2.
\]

It follows from Lemma \ref{lem:smallrank-1-epsilon1} that there are $1>\varepsilon_1>0,$ $q_{\varepsilon_1}$ and $N_1$ such that for any $n\geq N_1$ satisfying $q_n\geq q_{\varepsilon_1}$ and if $H$ is a simple (non-abelian) group satisfying the conditions of Theorem \ref{thm:main} with $r(H)<R_0$, we have $P(H^{F})<1-\varepsilon_1 .$

Let us pick a positive integer $M$ such that $(1-\varepsilon _1)^M<\varepsilon .$ (This is possible, since $0<\varepsilon_1<1.$) Further, pick an $l_0>Mr_0$ and choose $N$ so that $N\geq N_1$ for any $n\geq N$ we have $l_n\geq l_0.$ 
Let us consider the factorisation 
$$G_n^{F_n}=H_{1,n}^{F_n}\times H_{2,n}^{F_n}\times\cdots \times H_{m_n,n}^{F_n}\times (Z(G_n)/K_n)^{F_n}$$
(which is guaranteed by Proposition \ref{prop:group-factorising}), and recall from equation \eqref{eq:relation-rank-factors} that 
$$l_n=\sum_{i=1}^{m_n}r(H_{i,n}).$$
Observe that if $n\geq N$, we have $l_n>Mr_0,$ so either we have an $1\leq i\leq m_n$ such that $r(H_{i,n})\geq r_0$ or $r(H_{i,n})<r_0$ for all $1\leq i\leq m_n$ but $m_n>M.$  
By equation \eqref{eq:relation-P(G)-factors} we have 
$$P(G_n )= \prod_{i=1}^{m_n} P(H_{i,n}^{F_n}).$$

If there is a $1\leq i\leq m_n$ such that $r(H_{i,n})\geq r_0$ then by considering Corollary \ref{cor:simple-group-lower-bound}, Lemma \ref{lem:S_n-enough}, the inequality $q_n\geq l_n\cdot f(l_n)\geq r(H_{i,n})\cdot f(r(H_{i,n}))$ and our choices above all together, we have
$$P(G_n)<P(H_{i,n})\leq \left(1-\frac{f_{1,r(H_{i,n})}(q_n)}{f_{2,r(H_{i,n})}(q_n)}\right)+\frac{(q_n+1)^{2r(H_{i,n})}}{q_n^{2r(H_{i,n})}}\cdot \frac{6}{r(H_{i,n})^2}<\frac{\varepsilon}{2}+2\cdot \frac{\varepsilon}{4}=\varepsilon .$$
If we have $r(H_{i,n})<r_0$ then $m_n>M$ for all $1\leq i\leq m_n$, then we have 
$$P(G_n )<(1-\varepsilon_1)^{m_n}<(1-\varepsilon_1)^M<\varepsilon $$ by Lemma \ref{lem:smallrank-1-epsilon1}. 
This completes the proof of Theorem \ref{thm:main}.  

\section{An additive analogue of finite reductive group result}\label{s:finitereductivealgebra}
In this section, we consider an additive analogue of Theorems \ref{thm:fixedgroup} and \ref{thm:main} in Theorem \ref{thm:fixedliealgebra} and Conjecture \ref{conj:additiveanalogue}. We start this section by introducing the Fourier transform on $G^F$-invariant functions on the finite Lie algebra following \cite{Letellier05}. We also record several properties of this transform that will be used in the proofs of our results. These properties mirror certain well-known features of Deligne–Lusztig characters, and it is precisely this similarity that motivates the additive analogue of Theorems \ref{thm:fixedgroup} and \ref{thm:main}.

\subsection{Fourier transformation} \label{ss:frobeniustransform}  Let $\fg$ be the Lie algebra of a connected reductive group $G$ defined over an algebraically closed field $\overline{\mathbb{F}}_q$ for $q=p^m$.
Let us choose a non-trivial additive character $\psi: \mathbb{F}_q\rightarrow  \bC^\times$ and a  non-degenerate $G$-invariant symmetric bilinear form $\kappa: \fg^F\times \fg^F\rightarrow \mathbb{F}_q$. 
Let us assume that the characteristic $p$ is very good for $G$, then such an invariant form exists, see \cite[Proposition 2.5.12]{Letellier05}. With those ingredients, we may define the Fourier transform $ \mathcal{F}:\bC[\fg^F]^{G^F}\rightarrow \bC[\fg^F]^{G^F}$ as
\[
\mathcal{F}(\phi)(x):=\sum_{y\in \fg^F} \psi(\kappa(x,y))\phi(y) 
\]  
for $\phi \in \bC[\fg^F]^{G^F}$. Recall that here $\bC[\fg^F]^{G^F}$ is the space of $G^F$-invariant functions on $\fg^F$.

\subsubsection{} For each adjoint orbit $\mathcal{O}_i\in [\fg^F]$, let $1_{\mathcal{O}_i}^G$ denote the characteristic function of the adjoint orbit $\mathcal{O}_i \subseteq \fg^F$. The set ${\{1_{\mathcal{O}_i}^G\, |\, \mathcal{O}_i\in [\fg^F]\}}$ is a basis of $\bC[\fg^F]^{G^F}$. It is known that the Fourier transform preserves bases \cite[Proof of Proposition 5.2.22]{Letellier05}, so $\{\mathcal{F}(1_{\mathcal{O}_i}^G)\,|\, \mathcal{O}_i\in [\fg^F]\}$ is a basis of $\mathbb{C}[\fg^F]^{G^F}$.
We can think of this set as the additive analogue of the set of irreducible complex characters of $G^F$.  

\subsubsection{}

 Before computing $\mathcal{F}(1_{\mathcal{O}_i}^G)$, we introduce the Green function and Springer isomorphism. Furthermore, from \cite[\S3]{Letellier05}, we introduce the finite Lie algebra version of a Deligne-Lusztig character. This will help us compute $\mathcal{F}(1_{\mathcal{O}_i}^G)$ when $\mathcal{O}_i$ is regular semisimple. 

The Green function $Q_T^G$ is the function from unipotent elements of $G^F$ to $\mathbb{Z}$ defined by
\[
Q_{T}^{G}(u):=\Big(\mathrm{Ind}_{B^F}^{G^F} 1\Big) (u)=|\mathcal{B}_u^F|,  
\]
where $\mathcal{B}_u\subseteq G/B$ is the Springer fibre associated to the unipotent element $u\in G^F$. It is well-known that $Q_T^G(u)$ depends only on the $G$-orbit of $u$.

Since the characteristic $p$ is very good for $G$, 
a theorem of Springer implies that we have a $G$-equivariant isomorphism $\varpi$ from the nilpotent cone of $\fg$ to the unipotent variety of $G$, see \cite[\S 2.7.5]{Letellier05}.

Let $Y=Y_s+Y_n$ be the Jordan decomposition of $Y\in \fg^F$, where $Y_s$ is semisimple and $Y_n$ is nilpotent.
Let $\mathfrak{t}:=\mathrm{Lie}(T)$ for a maximal torus $T$ of $G$ and let us define the Harish-Chandra induction map as ${R_{\ft}^{\fg}: \bC[\ft^F]\rightarrow \bC[\fg^F]^{G^F}}$
 given by 
\begin{equation}\label{eq:algebradl}
R_{\ft}^{\fg} (f)(Y) = \frac{1}{|G_{Y_s}^F|} \sum_{\{g\in G^F\, | \, Ad_{g}Y_s\in \ft^F\}} Q_{T}^{G_{Y_s}}(\varpi(Y_n)) f(Ad_{g}(Y_s)),
\end{equation}
where  $G_{Y_s}$ is the centraliser of $Y_s$ in $G$. The value of $\mathcal{F}(1_{\mathcal{O}_i}^G)$ can be computed using this Harish-Chandra induction when $O_i$ satisfies a specific condition.

\subsubsection{} Note that there is a bijection between the $G^F$-adjoint orbits of Lie algebras of $F$-stable maximal tori of $G$ and the conjugacy classes of $W$. This follows from \cite[Proposition 3.3.3]{carter1985finite} and the fact that $Ad_g(\mathrm{Lie}(H))=\mathrm{Lie}(gHg^{-1})$ for a closed subgroup $H$ of G, cf. \cite[p. 229]{lehrer1992rational} or \cite[\S5.5.9]{Letellier05}. This is analogous to the statements in Section \ref{subs:F-stable-tori} for groups; and we will use the notation introduced there.

\subsubsection{Regular elements in $\ft[w]$}\label{sss:regularincartan} Let us denote by $\ft[w]$ the Lie algebra of $T[w]$. In this section, we define regular elements in $\ft[w]$. We say that $X\in \ft[w]$ is regular when $C_G(X)^\circ=T[w]$; this definition equivalent to the notion in \cite[Definition 2.6.10]{Letellier05}. Note that such regular elements exist since our very good assumption on the characteristic implies that the conditions in \cite[Lemma 2.6.13]{Letellier05} are satisfied.
 In addition,this gives that the stabiliser subgroup $C_G(X)$ of $X$ is connected from \cite[Proposition 2.6.18]{Letellier05}. It is easy to see that this is an additive analogue of regular semisimple element in $G$ since $\dim C_G(X)=r,$ cf. Definition \ref{def:regularsemisimple}. Using this observation, we can check the following.
\begin{lem}Let us consider a $F$-stable maximal torus $T$ of $G$.
Let $\ft=\mathrm{Lie}(T)$ and $X\in \ft$ a regular element. Then $X$ is not $G^F$-adjoint to any element of $\ft'=\mathrm{Lie}(T')$ for a $F$-stable maximal torus $T'$ such that $T$ and $T'$ are not $G^F$-conjugate. \end{lem}
\begin{proof}
Let us assume that $Ad_g(X)\in \ft'$ for some $g\in G^F$. Then $Ad_g(X)$ is again regular. This is because $C_G(Ad_g(X))$ is Levi subgroup containing $T'$ of $G$ from \cite[Lemma 2.6.13 (ii)]{Letellier05}. Furthermore, $\dim(C_G(X))=\dim(T)$, so $C_G(Ad_g(X))=T'$.
Then we have the following: $T'=C_G(Ad_g(X))=Ad_g(C_G(X))=Ad_g(T).$ However, $T$ and $T'$ are not $G^F$-conjugate, so this is a contradiction. Therefore, $X$ is not $G^F$-adjoint to any element of $\ft'$.
\end{proof}
Note that this property is analogous to the group case, where each regular semisimple element is in a unique maximal torus.

\subsubsection{A theorem of Kazhdan and Letellier} Now, we are ready to compute $\mathcal{F}(1_{\mathcal{O}_i}^G)$ when $\mathcal{O}_i$ is a regular semisimple adjoint orbit.
  Let $f_X: \ft^F\rightarrow \bC$ denote the character $\psi(\kappa|_{\ft}(-,X))$. The following result was proved by Kazhdan when the characteristic of the base field is large, and it was proved by Letellier that the characteristic being very good is sufficient.  

\begin{thm}{\cite[Theorem 7.3.3]{Letellier05}} Suppose $X\in \ft^F$ is regular and $Y\in \fg^F$. Let us denote the adjoint orbit of $X$ under $G^F$ action as $\mathcal{O}_X$. Then,  we have that
\[
\mathcal{F}(1_{\mathcal{O}_X}^G)(Y) =q^{|\Phi(\fg)^+|} R_{\ft}^{\fg}(f_X)(Y).
\] 
\label{t:KL}
\end{thm} 
Let us observe a property of this function that will be important to us.
\begin{cor}\label{coro:fourier} Let us assume that $X\in \ft^F$ is regular and $Y\in \fg^F$. If $Y$ is not $G^F$-adjoint to an element in $\ft^F$, then $\mathcal{F}(1_{\mathcal{O}_X}^G)(Y) = 0$.
\end{cor} 
\begin{proof}
    This is direct from the definition of $R_{\ft}^{\fg} (f)(Y)$ in equation \eqref{eq:algebradl}. More explicitly, if $\{g\in G^F\, | \, Ad_{g}Y_s\in \ft^F\}$ is empty, then $\mathcal{F}(1_{\mathcal{O}_X}^G)(Y) = 0$. For further discussion, please see \cite[Corollary 36]{KNWG}. 
\end{proof}

This corollary gives the following lemma, which is an additive version of Proposition \ref{prop:vanishingvalue}.

\begin{lem}\label{lem:vanishingvalueadditive}
Let $T$ and $T'$ be $F$-stable maximal tori of $G$ that are not $G^F$-conjugate. Then for each regular elements $X_1\in \ft=Lie(T)$ and $X_2\in \ft'=Lie(T')$, we have that $\mathcal{F}(1_{\mathcal{O}_{X_1}}^G)(X_2)=0$.
\end{lem}

\subsubsection{Additive analogue of Proposition \ref{prop:FstableTorifacts}}
Let us denote the Lie algebra of $T[w]$ by $\ft[w]$. The following statement is an additive analogue of a part of Proposition \ref{prop:FstableTorifacts}.

\begin{prop}\label{coro:centraliser}
	The quotient group $N_G(\ft[w])^F/C_{G^F}(\ft [w]^F)=N_G(\ft[w])^F/T[w]^F$ is isomorphic to $C_W(w)=\{x\in W\,|\, x^{-1}wx=w \}$.
\end{prop}

\begin{proof} 
The fact that $C_G(\ft [w]^F)=T[w]$ follows from \cite[Proposition 2.6.4]{Letellier05}. 
We may then obtain the desired result by showing that $N_G(\ft[w])/T[w] = N_G(T[w])/T[w]$, and subsequently applying \cite[Proposition 3.3.6]{carter1985finite} to $N_G(T[w])^F/T[w]^F$.
Let us show that  $N_G(T[w])=N_G(\ft[w])$ by using the existence of a regular element from \cite[Lemma 2.6.13]{Letellier05}.  It is easy to check that  $N_G(T[w]) \subset N_G(\ft[w])$ using $\mathrm{Lie}(gTg^{-1})=Ad_g(\mathrm{Lie}(T))$ again, and $N_G(T[w]) \supset N_G(\ft[w])$ comes from \cite[Lemma 2.6.16]{Letellier05} by considering a regular element $X$ in $\ft[w]$, i.e. $C_G(X)=T[w]$. This is because we have $Ad_g(X)\in \ft[w]$ for some $g\in N_G(\ft[w])$, and then \cite[Lemma 2.6.16]{Letellier05} implies that $g\in N_G(T[w])$ since the centre of $\ft[w]$ is itself.
\end{proof}

We can see that the situation so far is very similar to the one for groups. In this additive setting, our problem is again reduced to counting the number of adjoint orbits of regular elements in $\ft[w]$ for each $w\in W$. We will illustrate this in the next section. The following lemma will be useful. 


\begin{prop}\label{prop:lie-algebra-unique-contain}
	Let us consider the Lie algebra $\ft$ of a $F$-stable maximal torus $T$ of $G$. Then $\ft$ is the only Lie algebra of a maximal torus containing $\ft^F$. 
\end{prop}
\begin{proof} We have that $C_\fg(\ft^F)\subset C_\fg(X)$ for some regular element $X\in \ft$. This follows from  \cite[Lemma 2.6.17]{Letellier05}, since there is no root of $\fg$ with respect to $\ft$ satisfying $\alpha(X)=0$ for every $X\in \ft^F$. 
	That is, we have $C_\fg(\ft^F)\subset C_G(X)=T$. 
	Then from \cite[Proposition 2.6.4]{Letellier05}, we have
	\[
	C_\fg(\ft^F)\subset C_\fg(X)=\mathrm{Lie}(C_G(X))=\mathrm{Lie}(T)=\ft.
	\] 
Now, let $\ft^F\subset \ft'$ for the Lie subalgebra $\ft'$ of a maximal torus in $G.$ Then we have $	\ft' \subset C_\fg(\ft^F) \subset \ft.$
	Since $\ft$ and $\ft'$ are both Lie algebras of maximal tori, $\ft=\ft'$.
\end{proof}

\subsection{Proof of the additive analogues}\label{ss:proofdadditivecase} The proof of Theorem \ref{thm:fixedliealgebra} is almost the same as the proof of Theorem \ref{thm:fixedgroup} from Corollary \ref{coro:fourier}. We give a sketch of the proof. From now on, we assume that $\fg$ is a Lie algebra of a connected reductive group $G$ over $\overline{\mathbb{F}_q}$ with the Frobenius map $F$. Furthermore, we assume that the characteristic of $\mathbb{F}_q$ is very good and regular for $G$, and suppose that $q$ is sufficiently large.


\subsubsection{}  
We have the following lemma which is an additive analogue of Proposition \ref{prop:properties-counts-of-regular-semisimple}.
\begin{lem}  \label{lem:additivef_{i,G}}
Each $\ft[w_i]^F$ is a representative of a $G^F$-adjoint orbit of the Lie algebras of $F$-stable maximal tori in $\fg^F$. In addition, the number of adjoint orbits of regular elements in $\ft[w_i]^F$ over $G^F$ is $ \frac{g_i(q)}{c_i}$, where $g_i$ is a monic polynomial whose degree is the rank of $\fg$. Recall that $c_i=|C_W(w_i)|$. 
\end{lem}
\begin{proof}
The proof of the first statement is similar to the proof of \cite[Proposition 3.6.2]{carter1985finite}. For completeness, we add the proof here. Let $\ft$ and $\ft'$ be Lie algebras of $F$-stable maximal tori of $G$. Our claim is to show that $\ft$ and $\ft'$ are $G^F$-adjoint if and only if $\ft^F$ and $\ft'^F$ are $G^F$-adjoint. One direction is obvious. For the other, let us assume that $Ad_g(\ft^F)=\ft'^F$ for some $g\in G^F$. Then $\ft'^F\subset Ad_g(\ft)\cap \ft'$. Since $Ad_g(\ft)$ is the Lie algebra of the maximal torus $gTg^{-1},$ it follows from Proposition \ref{prop:lie-algebra-unique-contain} that $Ad_g(\ft)=\ft'.$ This proves the first statement.

Let $\Phi_i$ be the root system of $G$ over $T_i:=T[w_i]$. Then from \cite[Proposition 2.6.4]{Letellier05}, $X\in \ft[w_i]:=\ft_i$ is regular if and only if $d\alpha(X)\neq 0$ for all $\alpha\in \Phi_i$. 
We use the same logic as in the proof of \cite[Lemma 2.6.17]{Letellier05}, and therefore we need to consider the size of $\displaystyle \ft_i^F\setminus \underset{\alpha\in \Phi_i}{\bigcup}   {\mathrm{Ker}}(d\alpha)^F.$
Since $\dim({\mathrm{Ker}}(d\alpha))$ is strictly smaller than $\dim(\ft_i)$ (see \cite[Proof of Lemma 2.6.17]{Letellier05}), we can check that $\ft_i^F\setminus \underset{\alpha\in \Phi_i}{\cup} {\mathrm{Ker}}(d\alpha)^F$ is non-empty if $q>|\Phi_i|$.
 Furthermore, this implies that the number of regular semisimple elements in $\ft_i^F$ is monic with degree $\mathrm{rank}(\fg)$ since $|\ft_i^F|=q^{\mathrm{dim}(\ft_i)}$ (this equality is true because $\ft_i^F$ is a vector space).
 
 With Proposition \ref{coro:centraliser} and the fact that $C_G(X)$ is connected as illustrated in Section \ref{sss:regularincartan} for a regular element $X$ in $\ft_i$, we have that $Ad_{w'}X\in \ft_i$ is not equal to $X$ for any non-identity element $w'\in N_G(\ft_i)^F/T_i^F$. This is because if $Ad_{w'}(X)=X$, then $C_G(X)$ is not connected by \cite[Proposition 2.6.4 (i)]{Letellier05}. Therefore, we can see that the number of adjoint orbits of regular elements in $\ft_i^F$ over $G^F$ is $ \frac{g_i}{c_i}$.
\end{proof}

Now, we also have an additive analogue of Proposition \ref{prop:irresize} using the following lemma.

\begin{lem}{\cite[Corollary 2.5]{lehrer1992rational}}\label{lem:semisimple-contrib-Ffixed}
The number of semisimple adjoint orbits in $\fg^{F}$ under the action of $G^{F}$ is $q^r$. 
\end{lem}

\begin{rem}
Lehrer \cite[Corollary 2.5]{lehrer1992rational} considers $F$-stable semisimple orbits in $\fg_n$ under $G_n$-action, but there is a bijection between such $F$-stable semisimple orbits and semisimple adjoint orbits in $\fg_n^{F_n}$ under $G_n^{F_n}$-action. This is because $C_{G_n}(X)$ is connected for any semisimple element $X$ in $\fg_n$.
\end{rem}

\begin{prop}\label{prop:additiveirresize}
The number of adjoint orbits $|[\fg^{F}]|$ of $\fg^F$ under $G^F$-action is a monic polynomial whose degree is the rank of $\fg$.   \end{prop}
\begin{proof}Note that the size $|[\fg^F]|$ is the sum of semisimple orbits and non-semisimple orbits.
   Recall the fact that the number of nilpotent orbits in $C_\fg(t)^F$ is finite and independent of $q$ for any  semisimple element $t\in \fg^F$ \cite[Proposition 37]{KNWG}. Further, for $t$ a regular semisimple element $C_\fg(t)$ contains no nilpotent elements except the identity. From the fact that  and the number of non-regular semisimple orbits is degree at most $r-1$ from the proof of Lemma \ref{lem:additivef_{i,G}}, we conclude this proposition.
\end{proof}
An upper bound of $|[\fg^F]|$ will be considered later.

\subsubsection{Sketch of the proof of Theorem \ref{thm:fixedliealgebra}} The proof of Theorem \ref{thm:fixedliealgebra} is completely analogous to the proof of Theorem \ref{thm:fixedgroup} since every statement is similar has an analogue in the proof of the group version. Therefore, we only briefly explain how to prove the additive analogue. As $n$ approaches infinity, the size of finite fields increases, and so we assume that the finite field is sufficiently large. We will set a lower bound which has a similar form to equation \eqref{eq:lowerbound_content}, and then follow the steps in Section \ref{ss:lowerbound} and \ref{ss:proofmain2}. 
  From Lemma \ref{lem:vanishingvalueadditive} and Lemma \ref{lem:additivef_{i,G}}, we have the following lower bound:
\[
\frac{|\{(\mathcal{F}_n(\mathcal{O}_i),O_s) \in \mathfrak{F}_n\times [\fg_n^{F_n}]\,|\, \mathcal{F}_n(\mathcal{O}_i)(O_s)=0\}|}{|\mathfrak{F}_n|\times |[\fg_n^{F_n}]|} 
\geq 
\frac{\left|\left\{ (O_x ,O_y)\in [\fg_{n,rss}^{F_n}]\times [\fg_{n,rss}^{F_n}] \,\middle\vert\,  \substack{i\in \{1,2, \ldots , \tau \} \\O_x \cap \ft_n[w_i]^{F_n}\neq \emptyset \\ O_y \cap \ft[w_i]^{F_n}=\emptyset } \right\}\right|}
{  |[\fg_n^{F_n}]|^2} ,
\]
since $|\mathfrak{F}_n|= |[\fg_n^{F_n}]|$.
Here, $O_x$ is the adjoint orbit of an element $x$, $\fg_{n,rss}$ denotes the set of regular semisimple elements in $\fg_n$ and $[\fg_{n,rss}^{F_n}]$ the set of adjoint orbits under $G_n^{F_n}$-action in $\fg_n$. Then we can easily see that 
\[
\frac{\left|\left\{ (O_x ,O_y)\in [\fg_{n,rss}^{F_n}]\times [\fg_{n,rss}^{F_n}] \,\middle\vert\,  \substack{i\in \{1,2, \ldots , \tau \} \\O_x \cap \ft_n[w_i]^{F_n}\neq \emptyset \\ O_y \cap \ft[w_i]^{F_n}=\emptyset } \right\}\right|}
{  |[\fg_n^{F_n}]|^2}=\frac{|[\fg_{n,rss}^{F_n}]|^2}{|[\fg_n^{F_n}]|^2}-\frac{1}{|[\fg_n^{F_n}]|^2}\sum_{i=1}^\tau \frac{g_{n,i}^2}{c_i^2}
\]
where $g_{n,i}$ is the number of regular elements in $\ft_n[w_i]^{F_n}$. 

To finish the proof of Theorem \ref{thm:fixedliealgebra}, we need to establish a suitable lower bound of the right-hand side of the above equation. Recall that we  denote the rank of $G_n$ by $r$ (this does not depend on $n$ in the context of Theorem \ref{thm:fixedliealgebra}). We have two facts: $|\ft_n[w_i]^{F_n}|=q_n^r\geq g_{n,i}$ for all $i$ and the number of semisimple adjoint orbits in $\fg_n^{F_n}$ under $G_n^{F_n}$ action is $q_n^r$.  Using these facts, we get the following lower bound:
\[\begin{split}
\frac{|\{(\mathcal{F}_n(\mathcal{O}_x),O_y) \in \mathfrak{F}_n\times [\fg_n^{F_n}]\,|\, \mathcal{F}_n(\mathcal{O}_x)(O_y)=0\}|}{|\mathfrak{F}_n|\times |[\fg_n^{F_n}]|} 
& \geq  \frac{|[\fg_{n,rss}^{F_n}]|^2}{|[\fg_n^{F_n}]|^2}-\sum_{i=1}^\tau \frac{q_n^{2r}}{c_i^2\cdot q_n^{2r}}\\ &=\frac{|[\fg_{n,rss}^{F_n}]|^2}{|[\fg_n^{F_n}]|^2}-\sum_{i=1}^\tau \frac{1}{c_i^2}.
\end{split}\]
Now, let us consider the term $\frac{|[\fg_{n,rss}^{F_n}]|^2}{|[\fg_n^{F_n}]|^2}$. From \cite[Theorem 4.11]{lehrer1992rational}, we have that for a suitable polynomial $g$ the number $|[g_{n,rss}^{F_n}]|=g(q_n)$ for all $n$. (Here we use the assumption that the characteristic is regular.) In addition, we have that $\displaystyle g(q_n)=\sum_{i=1}^{\tau} \frac{g_i(q_n)}{c_i}.$ Thus by Lemma \ref{lem:additivef_{i,G}} and equation \eqref{eq:sumofreciprocalci} the polynomial $g(x)$ is monic of degree $r$.
Note that the polynomials $g(x)\in \mathbb{Z}[x]$ are listed in \cite[page 244]{lehrer1992rational} for each simple type. As the last component, we give an upper bound for $|[\fg_n^{F_n}]|$.

\begin{lem}\label{lem:upperbound-gnFnfixed}
There is a constant $C,$ independent of $n$, such that $|[\fg_n^{F_n}]|\leq q_n^r+C (q_n^r-g(q_n)).$
\end{lem}
\begin{proof}
The semisimple adjoint orbits contribute $q_n^r$ by Lemma \ref{lem:semisimple-contrib-Ffixed}. Consider an element $Y$ in a non-semisimple adjoint orbit. Its Jordan decomposition is $Y=Y_s+Y_n$, where $Y_n$ is a nilpotent element of $C_\fg(Y_s).$ In particular, if $Y_s$ is regular, then $Y_n$ is $0$. If $Y_s$ is not regular, then $Y_n$ is in a proper Levi subalgebra of $\fg.$ Up to $G$-action, there are only finitely many proper Levi subalgebras of $\fg,$ and by \cite[Proposition 37]{KNWG} each has only finitely many nipotent orbits. Thus there is a constant $C$ such that there are at most $C$ choices for $Y_n$ for each $Y_s.$ This completes the proof.  
\end{proof}
Observe that since $g(x)$ is monic with degree $r$, so is the upper bound $q_n^r+C(q_n^r-g(q_n)).$
Then we can finish the proof of Theorem \ref{thm:fixedliealgebra} by showing that $\underset{n \rightarrow \infty}{\mathrm{lim}}\frac{|[\fg_{n,rss}^{F_n}]|^2}{|[\fg_n^{F_n}]|^2}=1$. This can be proved using the same method as in the proof of Theorem \ref{thm:fixedgroup}, as carried out in Section \ref{ss:proofmain2}.

\subsubsection{A remark on an additive analogue of Theorem \ref{thm:main}}\label{sss:conjectureexplanation} 

Recall that we proposed an additive analogue of Theorem~\ref{thm:main} in Conjecture~\ref{conj:additiveanalogue}. 
Now, consider the sequence $\{\mathfrak{g}_n\}_{n=1}^\infty$ introduced in Conjecture~\ref{conj:additiveanalogue}. 
To prove this conjecture, we need to find a uniform upper bound for $|[\mathfrak{g}_n^{F_n}]|$ for all $n$, analogous to the group case, where we have
$|[H^F]| \leq q^r + 40q^{r-1}$
as in equation \eqref{eq:H-class-estimate-simple}. (Recall that this bound was used in Corollary \ref{cor:simple-group-lower-bound}, which in turn played a role in \ref{ss:proof-main}.) 
If such a uniform upper bound can be established, the same steps as used in the proof of Theorem~\ref{thm:main} lead to a proof of Conjecture~\ref{conj:additiveanalogue}. 
This motivates the following question.

\begin{ques}\label{que:missingstep-conj6}
Does there exist a polynomial $h(x) \in \mathbb{C}[x]$ such that 
\[
|\mathfrak{g}^F| \leq h(q)
\]
for any reductive Lie algebra $\mathfrak{g}$ defined over an algebraically closed field $k$ of positive characteristic, with Frobenius map $F = F_q$?
\end{ques}

\subsubsection{A problem about polynomiality on residue classes}\label{sssect:porc}
For a finite reductive group $G^F$, we say that two elements are of the same type if and only if their centralisers are conjugate. When the derived subgroup of $G$ is simply connected, the number of elements of the same type is a polynomial on residue classes; see \cite{deriziotis, BK23}. Motivated by this result, it is natural to consider an additive analogue of this problem. 

For a finite Lie algebra $\mathfrak{g}^F$, we again say that two elements $X$ and $Y$ are of the same type if and only if $C_G(X)$ and $C_G(Y)$ are conjugate. We denote by $\xi(\mathfrak{g})$ the set of types of $\mathfrak{g}$, and by $\mathcal{T}$ the type map from $\mathfrak{g}$ to $\xi(\mathfrak{g})$. We conjecture that the number of elements of the same type is polynomial on residue classes. We believe this would be an interesting problem to investigate. 

\begin{ques}\label{que:porc}
	Let $\mathfrak{g}$ be a connected reductive Lie algebra defined over an algebraically closed field $k$ of characteristic $p$. Let $F_{q}$ denote the Frobenius map for some $q=p^e$. Does there exist a positive integer $n_{\mathfrak{g}}$ depending on $\mathfrak{g}$ and polynomials 
	$h_{\xi,i}(x) \in \mathbb{C}[x]$ for $0\leq i\leq n_{\mathfrak{g}}-1$, such that for every $i$ (with $0\leq i\leq n_{\mathfrak{g}}-1$) and $\xi\in \mathcal{T}(X) $, and for each $q \equiv i \pmod{n_{\mathfrak{g}}}$ we have
	\[
	|\{X \in \mathfrak{g}^{F_q} \mid \mathcal{T}(X) = \xi\}| = h_{\xi,i}(q)?
	\]
\end{ques}


\begin{rem}
If Question~\ref{que:porc} has a positive answer, then we can remove the assumption that the characteristic $p_n$ is regular in Theorem~\ref{thm:fixedliealgebra} and Conjecture~\ref{conj:additiveanalogue}. This is because only finitely many polynomials contribute to the count $|[g_{n,rss}^{F_n}]|$.
\end{rem}

\bigskip
\noindent \textbf{Acknowledgments}		
GyeongHyeon Nam was supported by the National Research Foundation of Korea (NRF) grant funded by the Korea government (MSIT) (No. RS-2024-00334558) and Oscar Kivinen's Väisälä project grant of the Finnish Academy of Science and Letters.
We are very grateful to Michael J. Larsen for helpful advice.

\end{document}